\documentclass[a4paper]{article}\usepackage{geometry}

\usepackage[T1]{fontenc}
\usepackage[utf8]{inputenc}

\usepackage{graphicx,epsfig,lscape,float}    
\usepackage{psfrag,fancybox,authblk} 

\usepackage{tikz,pgfplots,circuitikz,pgfplotstable} 
\usetikzlibrary{shapes,arrows,matrix,calc}
\usepackage{tkz-euclide}
\usetkzobj{all}

\usepackage{subfigure}
\usepackage{caption}

\usepackage[linesnumbered,ruled]{algorithm2e}

\usepackage{amsmath,array,amssymb, amsfonts, latexsym,amsthm, mathrsfs, dsfont}

\usepackage{todonotes}
\usepackage{xspace}

\usepackage[parfill]{parskip}

\usepackage{multicol}   
\usepackage{enumerate}   
\usepackage{accents}  
\newcommand{\diag}{\mathbf{diag}}

\newcommand{\Rn}{\mathbb{R}^n}
\newcommand{\Rm}{\mathbb{R}^m}
\newcommand{\Rs}{\mathbb{R}^s}

\newcommand{\pcs}{PCS\xspace}
\newcommand{\pl}{PL\xspace}

\newcommand{\anf}{ANF\xspace}
\newcommand{\ad}{AD\xspace}
\newcommand{\ssn}{LSSN\xspace}
\newcommand{\lcp}{LCP\xspace}

\newcommand{\sig}{\mathcal{S}}

\newcommand{\suc}{{\mathrm{succ}}}
\newcommand{\pre}{{\mathrm{prec}}}

\newcommand{\al}{\alpha}

\newcommand{\ga}{\gamma}

\newcommand{\Om}{\Omega}

\renewcommand{\emptyset}{\varnothing}
\renewcommand{\subset}{\subseteq}

\newcommand{\numberset}[1]{\mathbb{#1}}
\newcommand{\R}{\numberset{R}}
\newcommand{\N}{\numberset{N}}

\newcommand{\chk}[1]{\check{#1}}
\newcommand{\roo}[1]{\mathring{#1}}

\newcommand{\cu}{{\chk u}}
\newcommand{\hu}{{\hat u}}

\newcommand{\rx}{{\roo x}}
\newcommand{\cx}{{\chk x}}
\newcommand{\hx}{{\hat x}}
\newcommand{\ry}{{\roo y}}
\newcommand{\cy}{{\chk y}}
\newcommand{\hy}{{\hat y}}
\newcommand{\rz}{{\roo z}}
\newcommand{\cz}{{\chk z}}
\newcommand{\hz}{{\hat z}}

\newcommand{\bd}{\partial}

\newcommand{\restr}[1]{\lower0.4ex\hbox{$\vert$}\lower0.7ex\hbox{ $\!_{#1}$ }}

\renewcommand{\epsilon}{\varepsilon}

\newcommand{\abs}{{\mbox{abs}}}

\theoremstyle{plain}
\newtheorem{proposition}{Proposition}[section]
\newtheorem{theorem}[proposition]{Theorem}
\newtheorem{corollary}[proposition]{Corollary}
\newtheorem{lemma}[proposition]{Lemma}
\newtheorem{definition}[proposition]{Definition}
\theoremstyle{remark}
\newtheorem{remark}[proposition]{Remark}
\newtheoremstyle{custom}{5pt}{5pt}{}{}{\bfseries}{:}{5pt}{}
\theoremstyle{custom}

\begin{document}


\title{An Open Newton Method for Piecewise Smooth Systems}


\author[1]{Manuel Radons}
\author[2]{Lutz Lehmann} 
\author[3,2]{Tom Streubel}
\author[4]{Andreas Griewank}
\affil[1]{Technical University of Berlin, Germany}
\affil[2]{Humboldt University of Berlin, Germany}
\affil[3]{Zuse Institute Berlin, Germany}
\affil[4]{School of Information Sciences Yachaytech, Ecuador}

\date{\scriptsize{\textit{radons@math.hu-berlin.de, llehmann@math.hu-berlin.de, streubel@zib.de, griewank@yachaytech.edu.ec}}}



\maketitle

\begin{abstract}\noindent
Recent research has shown that piecewise smooth (PS) functions can be approximated by piecewise linear functions with second order error in the distance to
a given reference point. A semismooth Newton type algorithm based on successive application of these piecewise linearizations was subsequently developed
for the solution of PS equation systems. For local bijectivity of the linearization
at a root, a radius of quadratic convergence was explicitly calculated in terms
of local Lipschitz constants of the underlying PS function. In the present work
we relax the criterium of local bijectivity of the linearization to local openness.
For this purpose a weak implicit function theorem is proved via local mapping
degree theory. It is shown that there exist PS functions $f:\R^2\rightarrow\R^2$ satisfying the weaker
criterium where every neighborhood of the root of $f$ contains a point $x$ such that
all elements of the Clarke Jacobian at $x$ are singular. In such neighborhoods
the steps of classical semismooth Newton are not defined, which establishes
the new method as an independent algorithm. To further clarify the relation between a PS function and its piecewise linearization,
several statements about structure correspondences between the two are proved. 
Moreover, the influence of the specific representation of the local piecewise linear models
on the robustness of our method is studied.
 An example application from cardiovascular mathematics is given.
\end{abstract}

 
 \vspace{4pt}
 

\section{Introduction}


The main objective of this article is the study of the generalized Newton's methods developed
in \cite{NewtonPL} for solving the equation \[F(x^*)\,=\,0\,, \]
where $F:\Rn\to\Rm$ is a continuous, piecewise smooth function. We are especially interested 
in situations that are singular in that there may be singular matrices both in the interior 
and the boundary of the Clarke Jacobian at $x^*$. That is, there may be singular 
limiting Jacobians in the sense of Facchinei and Pang, cf. \cite[p.627]{pang2003cp}.

Throughout we assume that $F$ can be computed by a finite program called \emph{evaluation procedure}.
An evaluation procedure is a composition 
of so-called elemental functions which make up the atomic constituents of more complex functions.
In the classical setting of \textit{algorithmic differentiation} (\ad) functions are admissible to a library $\Phi$ of elemental functions if they are at least once Lipschitz-continuously differentiable on their valid open domains. This condition is also called \textit{elemental differentiability} (ED), cf. \cite[p.23]{GriewankAD}.  
Common examples are: 
\[ \Phi\ :=\ \{ +, -, \ast, /, \sin, \cos, \tan, \cot, \exp, \log, \dots \}\; . \]
In our case, we will allow the evaluation procedure of $F: D\subseteq\R^n\rightarrow\R^m$ to contain, in addition to the usual smooth elementals, the absolute value $\abs(x) = |x|$, that is, our library is of the form \[\Phi_{\rm abs}\ :=\  \Phi\ \cup\ \{\abs \}\; .\]
Consequently, we can also handle the maximum and minimum of two values via the identities
\begin{align}\label{eq:identities}
\max(u,v)\, =\,\frac{u+v+|u-v|}{2},\qquad \min(u,v)\, =\,\frac {u+v-|u-v|}{2}\; .
\end{align}
We call the resulting class of functions {\em piecewise composite smooth} (PCS) and denote it by ${\rm span}(\Phi_{\rm abs})$. PCS functions are locally Lipschitz continuous and differentiable almost everywhere in the classical sense. 
Furthermore, they are differentiable in the sense of Bouligand and Clarke, cf. \cite{clarke}. 

It has been shown in \cite{griewank2013stable} that \pcs functions can be approximated locally 
with a second order error in the distance to a given reference point by piecewise linear (\pl) models.
In analogy to second order tangent and secant local linear approximations of smooth functions a tangent and a secant piecewise linearization mode were constructed.
We will denote the tangent mode piecewise linearization of a \pcs function $F:\Rn\to\Rm$
at a reference point $\rx$ by $\lozenge_{\rx}F$. The secant mode piecewise linearization 
at a pair of reference points $\cx,\hx$ will be denoted by  $\lozenge_{\cx}^{\hx}F$.
A key feature of these local models is that they vary Lipschitz-continuously with respect 
to perturbations of their reference point(s). This allowed to devise two generalized Newton's methods for \pcs functions $F:\Rn\to\Rn$, one using tangent mode piecewise linearizations (Algorithm \ref{algo:tNewton}) 	
		\begin{algorithm}
			\SetKwInOut{Input}{Step 0}
			\SetKwInOut{Output}{Step k}
			
			\Input{$x_0\in\Rn$ close to root of $F$}
			\Output{Compute tangent piecewise linearization $\lozenge_{x_{k-1}}F$ at $x_{k-1}$ and set $x_k$ to 
				solution of $\lozenge_{x_{k-1}}F=0$ with minimal distance to $x_{k-1}$}
			
			\caption{Tangent Mode Generalized Newton}
			\label{algo:tNewton}
		\end{algorithm}
	
	    	and one using their secant mode counterpart (Algorithm \ref{algo:sNewton}).		

	\begin{algorithm}
			\SetKwInOut{Input}{Step 0}
			\SetKwInOut{Output}{Step k}
			
			\Input{$x_0, x_1\in\Rn$ close to root of $F$}
			\Output{Compute secant piecewise linearization $\lozenge_{x_{k-1}}^{x_k}F$ at $x_{k-1}, x_k$ and set $x_{k+1}$ to 
				solution of $\lozenge_{x_{k-1}}^{x_k}F$ with minimal distance to $x_k$}
			
			\caption{Secant Mode Generalized Newton}
			\label{algo:sNewton}
		\end{algorithm}
	
	     In Algorithm \ref{algo:ssNewton} we present a schematic of semismooth Newton for locally Lipschitz functions $f:\Rn\to\Rn$ (\ssn), cf. \cite{kummer1988ssn}.
	Comparing it to Algorithms \ref{algo:tNewton} and \ref{algo:sNewton} 
	 one can see that the local \pl approximations take over the role that generalized (Clarke) Jacobians have in the classical algorithm. 
	Tangent mode \pl approximations and Clarke Jacobian reduce to local linear approximations (i.e., Jacobians) at nonsingular reference points. Hence, in the smooth case Algorithms \ref{algo:tNewton} and \ref{algo:ssNewton} coincide with Newton's method. The secant mode \pl approximations reduce to a form which is best described as a finite difference approximation of the Jacobian, cf. Section \ref{sec:basics}.
		 
	\begin{algorithm}
		\SetKwInOut{Input}{Step 0}
		\SetKwInOut{Output}{Step k}
		
		\Input{$x_0\in\Rn$ close to root of $F$}
		\Output{Select invertible element $A$ of Clarke Jacobian at $x_{k-1}$ and 
			compute $x^k:= A^{-1}[Ax^{k-1} -f(x^{k-1}) ]$}
		\caption{Semismooth Newton with Clarke Derivatives}
		\label{algo:ssNewton}
	\end{algorithm}

		    	It was proved in \cite{NewtonPL} that if the tangent mode piecewise linearization at a root $x^*$ of a \pcs function $F:\Rn\to\Rn$ is locally homeomorphic in some neighborhood of $x^*$, then close to this root the tangent and the secant mode of generalized Newton converge (quadratically resp. superlinearly) to $x^*$.
In this work we will, after a brief introduction to the theory of \pl functions 
and the piecewise linearization of \pcs functions (Section \ref{sec:basics}), relax this criterium to local openness of the tangent mode piecewise linearization at $x^*$ (Section \ref{sec:implicit-fkt}). 
A necessary condition for local openness of the tangent mode piecewise linearization at $x^*$ is
local openness of the underlying \pcs function at $x^*$. A sufficient condition is metric regularity of the underlying \pcs function at $x^*$ (Section \ref{sec:implicit-fkt}).

One rather intriguing feature of the local openness condition is that it includes cases where
the Clarke Jacobian at $x^*$ is singular, while the classical criterium for convergence of Algorithm \ref{algo:ssNewton} requires nonsingularity of the Clarke Jacobian at $x^*$, cf. \cite{Qi1993}. Moreover, the set of limiting Jacobians of $\lozenge_{x^*}F$ at $x^*$ is
merely a subset of the limiting Jacobians of $F$ at $x^*$. Hence $\lozenge_{\rx}F$ may be locally 
open even if some of the limiting Jacobians at $x^*$ are singular. 
 We will present an example problem for which this is the case in Section \ref{sec:counterexamples}. 

Another aspect that has to be considered comparing our methods to \ssn are the different
computational challenges that arise during the determination of the successive iterates. 
The question whether a given piecewise linear system has a solution is NP-complete. This follows immediately from Chung's proof of the NP-completeness of the general linear complementarity 
problem (\lcp) \cite{Chung1989} since the latter can be equivalently reformulated as a 
piecewise linear system called \textit{absolute value equation} (AVE); see, e.g., \cite{Prokopyev2007}  or \cite{Mangasarian2014}, also compare Section \ref{sec:ANF-perturbations}.
 In consequence, finding a root with minimal distance to a reference point of a piecewise linearization is a potentially hard computational task. 
Moreover, any perturbation of the reference point may cause an exponential number of
roots to appear or to vanish at arbitrary locations. 
We thus devote Section \ref{sec:unambiguous} to the study of conditions for the unambiguous 
computation of the next iteration point.

In Section \ref{sec:numerics} we present an example application from cardiovascular mathematics. We conclude with some final remarks. 


\section{Algorithmic Piecewise Linearization}\label{sec:basics}

In this section we will present the tangent and secant linearization modes which were developed in \cite{griewank2013stable} and further analyzed in \cite{NewtonPL} and \cite{PDODE}.  
\begin{figure}[htp]
	\centering
	\subfigure[Tangent mode linearization]{\includegraphics[scale=.9]{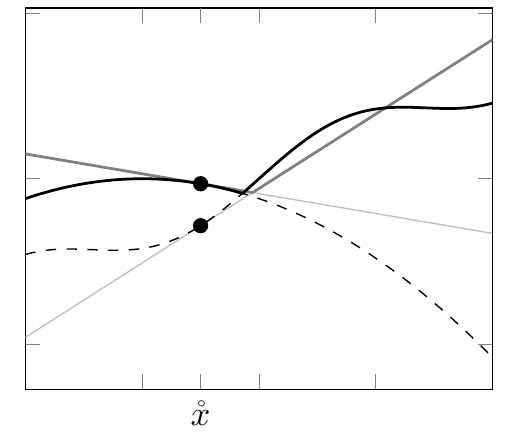}}
	\hspace{.2cm}
	\subfigure[Secant mode linearization]{\includegraphics[scale=.9]{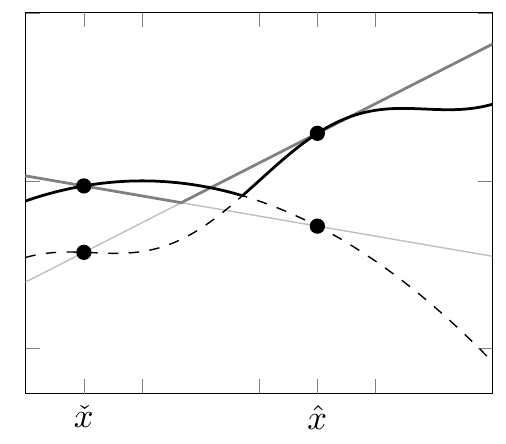}}
	\caption{Piecewise linearization modes}	
	\label{Fig:Modes}
\end{figure}\\

\subsection{Piecewise Linear Functions}
A continuous function $F:\mathbb R^n \rightarrow \mathbb R^m$ is called \textit{piecewise linear} if there exists a \textit{finite} set of affine functions $F_i(x)= A_ix+b_i$, such that $F$ coincides with an $F_i$ for every $x\in \mathbb R^n$ \cite[p.15ff]{scholtespl}. The $F_i$ are called \textit{selection functions}. If $F$ coincides with $F_i$ on a set $U\subset\mathbb R^n$, we say that $F_i$ is \textit{active} on $U$. Any piecewise linear function $F:\mathbb R^n \rightarrow \mathbb R^m$ admits a corresponding (nonunique) partition of $\mathbb R^n$ into \textit{nonempty} (and thus $n$-dimensional) convex polyhedra such that on each polyhedron in the partition exactly one selection function is active \cite[p.28]{scholtespl}.

A piecewise linear function $F:\Rn\to\Rn$ is called (locally) \textit{coherently oriented} if the linear parts of its selection functions all have the same nonzero determinant sign (on some some neighborhood $U\subset\Rn$). It is well known that for piecewise linear functions the properties of (local) coherent orientation and (local) openness are equivalent; cf \cite[Prop. 2.3.7]{scholtespl}.

Let the index set $I =\lbrace 1,...,k \rbrace$ of the selection functions be given.
According to \cite[Prop.2.2.2]{scholtespl} we can find subsets $M_1,...,M_l\subset I$ so that a scalar valued piecewise linear function $F$ can be expressed as 
\[F(x) = \max_{1\leq i\leq l}\min_{j\in M_i} F_j(x)\; .	\]	
This concept, which is called  \textit{max-min form}, naturally carries over to vector valued functions $F$, where we can find such a decomposition for every component of the image. Via \eqref{eq:identities}
the max-min form can be expressed in terms of $s\geq 0$ encapsulated absolute absolute value functions $| z_i |$, whose arguments $z_i$ are called {\em switching variables.} Observing that each $z_i$ is an  affine function of absolute values $|z_j|$ with $j < i$ and the independents $x_k$ for $k \leq n $, one arrives at an \textit{abs-normal} form (\anf)
\begin{eqnarray}
\begin{bmatrix} z  \\ y  \end{bmatrix} \;  = \; \begin{bmatrix} c \\  b \end{bmatrix} +   \begin{bmatrix} Z & L \\ J & Y   \end{bmatrix} \; 
\begin{bmatrix} x \\  |z|  \end{bmatrix} \label{eq:absnormal}  \; .
\end{eqnarray} 
Here the two vectors and four matrices specifying the function $F$ have the formats
$$ c \in \R^s, \;  Z \in \R^{s\times n}, \;  L \in   \R^{s\times s}, \;   b \in \R^m, \; J \in   \R^{m\times n}, \;   Y \in   \R^{m\times s}. $$ 
The matrix $L$ is strictly lower triangular so that for given $x$ the components of $z=z(x)$ and thus $|z|$ can be unambiguously  computed one by one.  In this sense, the ANF can be regarded as a matrix representation of the straight-line code of a piecewise linear function.
We give a simple $\R^2\rightarrow\R^2$ example to illustrate this:    
\[ 
\left[\begin{array}{c}
0 \\ 0 \\ 0 \\ \hline 0 \\ 0
\end{array}\right] + 
\left[\begin{array}{rr|rrr} 
1 & -1 & 0 & 0 & 0 \\ 
0 & 1 & 0 & 0 & 0 \\
1 & 0 & 0 & -1 & 0 \\ 
\hline 
1 & 0 & 1 & 0 & 1 \\ 
0 & 1 & 0 & 0 & 0 
\end{array}\right]
\left[\begin{array}{c} 
x_1 \\ 
x_2 \\
\hline 
\lvert z_1 \rvert \\ 
\lvert z_2 \rvert \\ 
\lvert z_3 \rvert
\end{array}\right] = 
\left[\begin{array}{c} 
z_1 \\
z_2 \\ 
z_3 \\ 
\hline
y_1 \\ 
y_2 
\end{array}\right] 
\quad \quad 
\begin{array}{l}
z_1= x_1 - x_2 \\
z_2=x_2 \\
z_3 = x_1 -|z_2| \\
F_1(x_1,x_2) = x_1+\lvert z_1\rvert + \lvert z_3 \rvert \\
F_2(x_1,x_2) = x_2 
\end{array}
\]

For more detailed information on this structure and its properties we refer to \cite{griewank2013stable}, and \cite{griewank2014abs}.

\subsection{Piecewise Linearization}

We now want to compute an incremental approximation $\Delta y = \Delta F(\rx;\Delta x)$ to $F(\rx + \Delta x)-F(\rx)$
at a given $\rx$ and for a variable increment $\Delta x$.
Assuming that all functions other than the absolute value
are differentiable, we introduce the propagation rules
for binary arithmetic operations, unary smooth functions $\varphi$ and the
absolute value
\begin{equation}\label{Eqn:TangentLinearization}
	\begin{array}{llll}
		\Delta v_i=\Delta v_j\pm \Delta v_k  & \mbox{for }  \mathring{v}_i=\mathring{v}_j\pm \mathring{v}_k\; , &\\
		\Delta v_i=\mathring{v}_j\ast\Delta v_k+\Delta v_j\ast \mathring{v}_k  & \mbox{for } \mathring{v}_i=\mathring{v}_j\ast \mathring{v}_k\; , &  \\
		\Delta v_i=\mathring{c}_{ij}\Delta v_j & \mbox{for } \mathring{v}_i=\varphi_i(\mathring{v}_j),\;\;\varphi_i\neq \abs\;,  \;\;\; \mbox{with } \mathring{c}_{ij}=\varphi_i'(\mathring{v}_j)\\
		\Delta v_i=\abs(\mathring{v}_j+\Delta v_j)-\abs(\mathring{v}_j)  &  \mbox{for } \mathring{v}_i=\abs(\mathring{v}_j)\; .
	\end{array}
\end{equation}
Whenever $F$ is ED and thus globally differentiable (i.e., there are no $\abs$ calls
in the evaluation procedure) we get $\Delta y = F'(\rx)\Delta x$,
where $F'(\rx)\in \R^{m\times n}$ is the Jacobian of $F$ at $\rx$.

The propagation rules \eqref{Eqn:TangentLinearization} define the
\emph{tangent mode piecewise linear approximation} of $F$ at a certain point $\rx$. However,
there are applications of piecewise linearization (especially concerning ODE integration)
where one wants to consider approximations of $F$ based on secants in the following sense.
Given two points $\check x, \hat x$ in the domain, the secant approximation will be a piecewise linear approximation of $F$ that is exact at the points $\cx$ and $\hx$ and reduces to the secant over the line through $\cx,hx$ for smooth functions complying with ED. For the notation of the
secant approximation the increment 
$\Delta F(\Delta x)=\Delta F(\cx,\hx;\Delta x)$
will be relative to the reference point 
$\mathring x = (\check x + \hat x)/2$
and the reference value $\mathring F= (F(\check x) + F(\hat x))/2$, that is
\begin{equation}\label{Eqn:SecantApproximation}
	F(x)\ \approx\ \mathring F + \Delta F(\cx,\hx; x-\mathring x)\; .
\end{equation}


The actual definition of the secant approximation in \eqref{Eqn:SecantApproximation} is again given recursively
over the composition steps which allows a direct translation into an
AD like algorithm. 
The differentiation rules for the secant mode are mostly similar 
to the ones for the tangent mode in \eqref{Eqn:TangentLinearization}, 
with the following modifications
\begin{itemize}
  \item the tangent slope $\mathring{c}_{ij}$ in the third line 
	has to be replaced by the secant slope
	\begin{equation}
	\mathring c_{ij} = \left\lbrace\begin{array}{r l} \varphi'_i(\mathring v_j) & \mbox{ if }\check v_j = \hat v_j \vspace{0.2cm} \\ \dfrac{\hat v_i - \check v_i}{\hat v_j - \check v_j} &\mbox{ otherwise}\end{array}\right. \; . \label{eqn:secant_cij}
	\end{equation}
  \item in the last line the offset has to be taken against the 
    reference value,
	\begin{equation}
		\Delta v_i=\abs(\mathring{v}_j+\Delta v_j)-\mathring{v}_i,  ~~  \mbox{ where now } \mathring{v}_i= \frac{\check v_i + \hat v_i}2 = \frac{|\check v_j|+|\hat v_j|}2.
	\end{equation}
  
\end{itemize}
The secant approximation depends to an even greater degree on the given composition of $F$ by elementary functions,
a different evaluation procedure for the same function may give a different secant approximation. 

Similar to the tangent mode, if $F$ complies with ED and is thus again globally differentiable, the secant approximation is linear in the $x$ increment.
We can write this as 
\begin{equation}\label{eq:secantpropmatrix}
\Delta F(\hat x, \check x; \Delta x) 
= \nabla_{\check x}^{\hat x}F(\cx,\hx) \cdot \Delta x,
\end{equation} and call \(\nabla_{\check x}^{\hat x}F\) the \emph{propagated secant matrix}. 
The entries of the propagated secant matrix are calculated using the rules in table \eqref{Eqn:TangentLinearization} while employing the secant slope formula \eqref{eqn:secant_cij} for \(\mathring c_{ij}\). 
Moreover, if $\check x = \hat x$, we obtain $\Delta F(\check x, \hat x;\Delta x) = \Delta F(\rx,\Delta x)$, i.e., in this case the tangent and secant linearizations are identical.

A complete discussion of this topic can be found in \cite[Sec. 7]{griewank2013stable}. Additionally, a division-free and thus numerically stable implementation is detailed in \cite[Sec. 6]{NewtonPL}. 


%

\subsubsection*{Approximation quality}

Denote by $\lozenge_{\rx} F(x) \equiv F(\rx) + \Delta F(\rx; x - \rx)$ the tangent mode piecewise linearization of $F$ at the reference point
 $\rx\in\mathbb R^n$ and by $\lozenge_{\cx}^{\hx}F(x) \equiv \tfrac 12(F(\hx) + F(\cx)) + \Delta F(\hx, \cx; x - \tfrac 12(\hx + \cx))$ the secant mode piecewise linearization at the reference points  $\cx,\hx\in\mathbb R^n$. The two results from \cite{NewtonPL} that we will draw on frequently, are: 
\begin{proposition}\label{prop:approx}
	Suppose $\cx, \hx, \cy, \hy, \cz, \hz \in \R^n$ are restricted to a sufficently small closed convex neighborhood $K \in \R^n$ where the evalution procedure  for $ F : \R^n \mapsto \R^m $ is well defined. Then there exists a Lipschitz constant $\gamma_F$ such that it holds
	\begin{enumerate}
		\item For all $x\in K$ we have:   
		\begin{eqnarray*}
			\|F(x)-   \lozenge_{\cx}^{\hx}  F(x) \| &\; \leq \; & \tfrac{1}{2}  \gamma_F \|x - \hx\|\|x - \cx\| \; ,  \\[0.2cm]  
			\|F(x)-   \lozenge_{\rx} F(x) \|  &\; \leq \; &\tfrac{1}{2}  \gamma_F \|x -\rx\|^2\; .
		\end{eqnarray*}
		
		\item For all $x\in\mathbb R^n$ we have: 
		{}
		%
		\begin{align*}
		\|   \lozenge_{\cz}^{\hz}  F(x) - \lozenge_{\cy}^{\hy}  F(x) \| &\; \leq \; \gamma_F \; 
		\max \left [\| \hz-\hy \| \max ( \|x-\cy\|,  \|x-\cz\| ),  \right .  \\  
		&  \hspace*{2.1cm}  \left . \| \cz-\cy \| \max ( \|x-\hy\|,  \|x-\hz\| ) \right ] \; , \\[0.2cm]
		\|   \lozenge_{\rz}  F(x) - \lozenge_{\ry}  F(x) \|  &\; \;  \, \leq \quad 
		\gamma_F \quad \;  \,  \|\rz-\ry\|  \max (\|x-\ry\|,  \|x-\rz\|  ) \; .  &
		\end{align*}

	\end{enumerate} 
\end{proposition} 
Note that in \cite{NewtonPL} it was also described how to calculate an upper bound for $\gamma_F$. 

\subsubsection*{Representation in abs-normal form}

We will represent all our piecewise linearizations in \anf. For the tangent mode the \anf can be computed as follows.
Let \(F:\,\R^n \to \R^m\) be a \pcs function, then there are smooth vector valued evaluation procedures \(G:\,\R^{n} \times \R^s \to \R^s\) and \(\tilde F:\,\R^n \times \R^s \to \R^m\) of operations in \(\Phi\) and hence complying to ED so that
\begin{align}
  \begin{aligned}
    z \, =\, G(x, |z|)\qquad\text{and}\qquad
    F(x) \, =\, \tilde F(x, |z|)\,,
  \end{aligned} \label{eq:F_nonsmoothANF} 
\end{align}
where the partial derivative matrix \(\frac\partial{\partial w} G(x, w)\) is always of strictly lower triangular form. Applying an order \(1\) Taylor Expansion on \(\tilde F\) and \(G\) yields:
\begin{align}
  \begin{bmatrix}
    z \\
    y
  \end{bmatrix} \ =\ 
  \begin{bmatrix}
    G(\mathring x, |\mathring z|) \\
    \tilde F(\mathring x, |\mathring z|)
  \end{bmatrix}\, +\, 
  \begin{bmatrix}
    \frac\partial{\partial x} G(\mathring x, |\mathring z|) & \frac\partial{\partial |z|} G(\mathring x, |\mathring z|) \\
    \frac\partial{\partial x} \tilde F(\mathring x, |\mathring z|) & \frac\partial{\partial |z|} \tilde F(\mathring x, |\mathring z|)
  \end{bmatrix}\, \cdot\, 
  \begin{bmatrix}
    x - \mathring x \\
    |z| - |\mathring z|
  \end{bmatrix}\,. \label{eqn:PLANF}
\end{align}
Now system \eqref{eqn:PLANF} is a piecewise linear operator mapping \(x\) to \(y = \lozenge_{\rx} F(x)\) via intermediate vector switching variables \(z\).
Setting \(Z \equiv  \frac \partial{\partial x} G\), \(L \equiv  \frac \partial{\partial |z|} G\), \(J \equiv  \frac \partial{\partial x} F\), \(Y \equiv  \frac \partial{\partial |z|} F\) as well as \(c \equiv G(\mathring x, |\mathring z|) - Z\rx - L|\mathring z|\) and \(b \equiv \tilde F(\mathring x, |\mathring z|) - J\rx - Y|\mathring z|\), system \eqref{eqn:PLANF} becomes the \anf \eqref{eq:absnormal}. 
 
Similar to the tangent case, an \anf-representation of secant piecewise linearizations can be computed by replacing the Jacobi matrices with
the propagated secant matrices of \eqref{eq:secantpropmatrix} at the 
combined base points $(\cx,\cz)$ and $(\hx,\hz)$. 
The coefficient vectors \(c\), \(b\) and four matrices \(Z, L, J\) and \(Y\) of the \anf are then obtained from
\begin{align*}
  c &\equiv \mathring G - L\cdot|\mathring G| \in \R^s, &  b &\equiv \mathring F - Y\cdot|\mathring F| \in \R^m, &
  [Z\,\,\, L] &= \nabla_{\check x, |\check z|}^{\hat x, |\hat z|} G, & [J\,\,\, Y] &= \nabla_{\check x, |\check z|}^{\hat x, |\hat z|} F\,,
\end{align*}
where \(\mathring x = (\check x + \hat x)/2\) and \[\mathring G\, \equiv\, 
\frac{G(\check x, |\check z|) + G(\hat x, |\hat z|)}{2}
 \qquad\mathring F\, \equiv\, 
 \frac{F(\check x, |\check z|) + F(\hat x, |\hat z|)}{2}\] .
\begin{remark}(Dimensions of the \anf)\label{rem:ANF-dim}
  Let \(F:\Rn\to\Rm\) be a \pcs function with \(s\) absolute values occurring in its evaluation procedure. Then the two vectors and four matrices of its tangent or secant mode piecewise linearization in \anf 
  have the formats
  \[ c \in \R^s, \;  Z \in \R^{s\times n}, \;  L \in   \R^{s\times s}, \;  
  b \in \R^m, \; J \in   \R^{m\times n}, \;   Y \in   \R^{m\times s}, \]
  irrespective of the choice of the reference point(s). In the sensitivity analysis of Section \ref{sec:ANF-perturbations} this key fact will allow us to analyze successive piecewise linearizations of a PCS function as perturbations of a single \anf.
\end{remark} 


\section{Generalized Newton methods by piecewise linearization}

We will now present and analyze the piecewise differentiable Newton's methods proposed in \cite{NewtonPL}. The merit of these methods is the fact that they impose no strong differentiability requirements but need only piecewise differentiability at the root which is to be computed. 

\begin{definition}(Newton operator)
Let $F\in {\rm span}(\Phi_{\rm abs})$ and $x^*$ be an isolated root of $F$ 
in an open neighborhood ${\cal D}$. The Newton step for $F$
is definable on ${\cal D}$ in tangent or secant mode if the piecewise linear equation
$\lozenge_{\rx} F(x)=0$ resp. $\lozenge_{\cx}^\hx F(x)=0$ has at least 
one root for all $\rx\in U$ resp. $\cx,\hx\in {\cal D}$.

Then the Newton operator is defined in tangent mode as 
\[
N(\rx)=\arg\min\{\|x-\rx\|:\lozenge_{\rx} F(x)=0\}\,,
\]
and in secant mode as
\[
N(\rx)=\arg\min\{\|x-\rx\|:\lozenge_{\cx}^\hx F(x)=0\}\,,
\]
where $\rx=\frac12(\cx+\hx)$.
\end{definition}

Employing coherent orientation to replace the invertibility of the 
Jacobian in the smooth Newton method as a regularity condition,
we will proceed to show that close to a root $x^*$ of $F$
\begin{itemize}
 \item the Newton step can be defined
 \item the Newton step stays close to $x^*$
 \item the tangent mode Newton method converges quadratically
 and the secant mode Newton method converges with the golden mean as order.
\end{itemize}

For most of the following results up to the last it is sufficient to consider 
the secant mode piecewise linear approximation of $F$, as results for the tangent 
mode can be obtained by setting $\cx=\hx=\rx$.

The general bound for the distance between piecewise linearizations with different basis points 
$\cz,\hz$ and $\cy=\hy=\ry$
can be refined to a tighter bound in case one expresses that bound exclusively in terms of
distances to $\ry$. 
\begin{lemma}(Refinement for the distance between piecewise linearizations)
  Let the second order constant $\ga_F$ be valid on some convex set $U$ and $x,\ry,\cz,\hz\in U$.
  Then 
  \[
  \|\lozenge_{\cz}^\hz F(x)-\lozenge_{\ry} F(x)\|
  \le \ga_F \Bigl(\max(\|\hz-\ry\|,|\cz-\ry\|)\,\|x-\ry\|+\frac12\|\hz-\ry\|\,\|\cz-\ry\|\Bigr)
  \]
\end{lemma}
\begin{proof}
 Select some $N\in \N$ and consider the subdivision of the segments $[\cz,\ry]$ and $[\hz,\ry]$.
 by $\cu_k=\ry+\frac{k}{n}(\cx-\ry)$ and $\hu_k=\ry+\frac{k}{n}(\hx-\ry)$
 Then by Proposition \ref{prop:approx} 
 \begin{align}
  \|\lozenge_{\cu_{k+1}}^{\hu_{k+1}} F(x)-\lozenge_{\cu_{k}}^{\hu_{k}} F(x)\|
  &\le \ga_F\max\left(\begin{gathered}
    \|\hu_{k+1}-\hu_k\|\,\max(\|x-\cu_{k+1}\|,\|x-\cu_k\|)\\
    \|\cu_{k+1}-\cu_k\|\,\max(\|x-\hu_{k+1}\|,\|x-\hu_k\|)            
    \end{gathered}\right)
  \notag\\
  &\le\ga_F \max\left(\begin{gathered}
    \tfrac1N\|\hz-\ry\|\,\bigl(\|x-\ry\|+\tfrac{k+1}N\|\cz-\ry\|\bigr)\\
    \tfrac1N\|\cz-\ry\|\,\bigl(\|x-\ry\|+\tfrac{k+1}N\|\hz-\ry\|\bigr)            
    \end{gathered}\right)
  \notag\\
  &=\ga_F\left(\tfrac1N\max(\|\hz-\ry\|,\|\cz-\ry\|)\|x-\ry\|+\tfrac{k+1}{N^2}\|\hz-\ry\|\,|\cz-\ry\|\right)
 \end{align}
 Summation and taking the limit for $N\to\infty$ results in the claim.
\end{proof}

\subsection{Mapping Degree}

The following definitions and facts can be found, e.g., in \cite[p. 111ff]{ruiz2009deg}. Let $f:\mathbb R^n\rightarrow\mathbb R^n$ be a continuous function, $\Omega\subset\mathbb R^n$ a bounded domain, and let $y\in\mathbb R^n\setminus f(\partial\bar\Omega)$, where $\bar\Omega$ is the closure of $\Omega$ and $\partial\bar\Omega$ denotes the boundary of $\bar\Omega$. We say $y$ is a \textit{regular value} of $f\restr{\Omega}$ if either $(f\restr{\Omega})^{-1}(y)=\emptyset$ or if the differential $D_xf$ of all $x\in f^{-1}(y)$ exists and is invertible. The \textit{local (Brouwer) degree} of $y$ on $\Omega$ is denoted by $\operatorname{deg}(f,\Omega,y)$. We will need the following two of its properties:
\begin{enumerate}
	\item $\operatorname{deg}(f, \Omega,y) = \sum_{x\in \left(f\restr{\Omega}\right)^{-1}(y)}\operatorname{sign}[\det(D_xf)]$ -- which especially implies that \[\left(f\restr{\Omega}\right)^{-1}(y)\ne\emptyset\quad\quad \text{ if }\quad\quad \operatorname{deg}(f, \Omega,y)\ne 0\,.\] 
	\item \textit{Nearness property:} Let $y, y'$ be regular values of $f\restr{\Omega}$. If \[\operatorname{dist}(y,y')\ <\ \operatorname{dist}(y,f(\partial\bar\Omega))\,,\] then $\operatorname{deg}(f, \Omega,y)=\operatorname{deg}(f, \Omega,y')$.
	\item The regular values of $f$ are dense in the codomain.
\end{enumerate}
Also note that for a piecewise affine function $F:\R^n\rightarrow\R^n$ the set $F^{-1}(y)$ of preimages of any regular value $y$ of $F$ is finite and discrete. 

\subsection{Weak Implicit Function Theorem and Convergence}\label{sec:implicit-fkt}

Denote by $B_r(x)$ a full-dimensional ball of radius $r$, centered at $x$.

\begin{lemma}\label{lem:local-surj}
 Let $F\in {\rm span}(\Phi_{\rm abs})$ and $x^*$ be an isolated root of $F$ 
in an open neighborhood ${\cal D}$. If $\lozenge_{x^*}F$ is coherently oriented 
on ${\cal D}$, then there exist radii $r,R>0$ 
so that for any
$\al\in(0,1)$ and $\cx,\hx\in B_{(1-\al)R}(x^*)$ a ball about the root 
is contained in the image of the linearization at $\cx,\hx$,
\[
B_{\al r}(0)\subset \lozenge_\cx^\hx F({\cal D})
\]
\end{lemma}

\begin{proof}
  Set $R_{\cal D}={\rm dist}(x^*,\bd{\cal D})$ and for any $R\in (0,R_{\cal D})$
  \[
  r=\rho(R)={\rm dist}[0,\lozenge_{x^*}F(\bd B_R(x^*))]
  \]
  By coherent orientation of $\lozenge_{x^*}F$, $x^*$ is isolated so that for $R>0$ small enough one will get $r>0$. 
  Increasing $R$ might bring the 
  sphere $\bd B_R(x^*)$ close to a different root, decreasing $r$ towards zero in consequence.
  Pick an $\tilde R>0$ with $r=\rho(\tilde R)>0$.
  Let the second order constant $\ga_F$ of $F$ be valid on $\cal D$ and set 
  \[
  R=\min\left(\tilde R,\frac{2r}{3\ga_F\tilde R}\right)
  \]
  We now show that this pair of $r$ and $R$ satisfies the claim of the lemma. To that end consider
  the set \[\Om\,\equiv\, (\lozenge_{x^*}F)^{-1}\left[B_r(0)\cap \lozenge_{x^*}F(B_{\tilde R}(x^*))\right]\] which is open as the coherent 
  orientation implies the openness of $\lozenge_{x^*}F$. By construction we find $\|\lozenge_{x^*}F(z)\|=r$
  for all $z\in\bd\Om$.
    
  Fix $\alpha\in(0,1)$ and $\cx,\hx\in B_{(1-\al)R}(x^*)$. Then for any $z\in\bd\Om$
  \begin{align}
  \|\lozenge_\cx^\hx F(x)-\lozenge_{x^*} F(x)\|
  &\le \ga_F\left( (1-\al)R\,\tilde R+\frac12 (1-\al)^2R^2\right)
  \notag\\
  &\le \frac32(1-\al)\ga_F\tilde R\,R\le (1-\al)r
  \end{align}
  For the linear homotopy $H_t=(1-t)\lozenge_{x^*} F+t\lozenge_\cx^\hx F$, $t\in[0,1]$, we thus know that none of the 
  pre-images $H_t^{-1}(B_{\al r}(0))$ will cross $\bd\Om$. Thus for any regular value $y\in B_{\al r}(0)$ the 
  Brouwer degree satisfies 
  \[
  \deg(H_t,\Om,y)=\deg(\lozenge_{x^*} F,\Om,y)\ne 0\;,
  \]
  which implies that at least one solution of $H_t(x)=y$ exists for any $y\in B_{\al r}(0)$ and $t\in[0,1]$. Thus
  \[
  B_{\al r}(0)
  \subset \lozenge_\cx^\hx F(\Om)
  \subset\lozenge_\cx^\hx F(B_{\tilde R}(x^*))
  \subset \lozenge_\cx^\hx F({\cal D})\;,
  \]
  which completes the proof.
\end{proof}
\begin{remark}\label{rem:regular-root}
	Lemma \ref{lem:local-surj} states that $0$ is contained in the image of the perturbed piecewise linear model. This does not mean that it is a regular value. In fact, an exponential 
	number [in $\mathcal O(2^s)$] of nonempty polyhedra may be mapped to $0$ by arbitrarily small perturbations of 
	the piecewise linearization at $x^*$. Hence, in the worst case one has to compute 
	the point on these polyhedra with minimal distance to the reference point of the
	current linearization. Further, even if $0$ is a regular value, arbitrarily small
	perturbations of the reference point of the piecewise linearization may add or remove 
	an exponential even number of roots [again in $\mathcal O(2^s)$], cf. Figure \ref{fig:pert}
\end{remark}
As the root $x^*$ is isolated, moving a small distance away from $x^*$ will increase the norm of
$\lozenge_{x^*} F(x)$ linearly in $\|x-x^*\|$ with some positive slope. This idea generalizes 
into the concept of metric regularity. 
\begin{definition}(Metric regularity \cite[p. 20]{mord2006varana})
 A locally Lipschitz continuous function $F:\R^n\rightarrow\R^m$ is called \textbf{metrically regular} at $x^*\in\R^n$ if there exist a constant $c>0$, and neighborhoods $\mathcal D\subset\R^n$ of $x^*$, and $\mathcal V\subset\R^m$ of $F(x^*)$ such that 
 \begin{align*}
  \lVert x-F^{-1}(y)\rVert \le c\lVert y-F(x)\rVert\qquad \forall\ x\in\mathcal D,\; y\in\mathcal V\; .
 \end{align*}
 We say \(F\colon\R^n\to\R^m\) is metrically regular on an open
neighborhood \(\mathcal{D}\subset\R^n\) if it is metrically regular at
all points \(x\in\mathcal{D}\).
\end{definition}
\begin{remark}
	Let $F:\R^n\rightarrow\R^n$ be piecewise affine and coherently oriented on a connected set $\mathcal D\subseteq\R^n$. Denote by $F_1,\dots, F_k$ the affine selection functions which are active on $\mathcal D$ and by $A_1,\dots, A_k$ their (by hypothesis invertible) linear parts. Then $F$ is metrically regular on $\mathcal D$ with associated constant 
	\[c\ :=\ \max\{\Vert A^{-1}_1\Vert,\dots,\Vert A^{-1}_k\Vert \}\; .\]
\end{remark}
This allows to further quantify the previous result and
strengthen its claim.

\begin{proposition}(Contractivity)\label{prop:contractivity} Let $F\in {\rm span}(\Phi_{\rm abs})$ and 
$x^*$ be an isolated root of $F$ in an open neighborhood ${\cal D}$ so that 
$\lozenge_{x^*}F$ is coherently oriented on ${\cal D}$. Let further $\lozenge_{x^*}F$ be 
strongly metrically regular on $B(x^*,\tilde R)\subset {\cal D}$, $\tilde R>0$,
in that for some $c>0$ and all $x\in B(x^*,\tilde R)$ we have
\[
\|x-x^*\|\le c\|\lozenge_{x^*}F(x)\|\,.
\]
Then for $$R=\min\left(\frac{\tilde R}3, \frac2{3c\ga_F}\right)$$
the Newton operator both in tangent and secant mode of the piecewise linearization 
 is contractive on $B(x^*,R)$.
\end{proposition}
\begin{proof} 
 As for any $z\in\bd B_{\tilde R}(x^*)$ the piecewise linearization 
 has a lower bound \[c\cdot\|\lozenge_{x^*}F(x)\|\ \ge\ \tilde R\,,\]
 we obtain from the previous lemma that $\lozenge_\cx^\hx F(x)=0$ will have a
 root $x\in B_{\tilde R}(x^*)$ if the radii are chosen according to
 $r= c^{-1}\tilde R$, $R\le \min(\tilde R,\frac2{3c\ga_F})$ and $\cx,\hx\in B_R(x^*)$.
 
 For any root $x$ of $\lozenge_\cx^\hx F$ inside $B_{\tilde R}(x^*)$ we find further that
 by metric regularity and stability of the piecewise linear approximations
\begin{align}\label{eq:Newton:basic}
\|x-x^*\|
&\le c\|\lozenge_{x^*}F(x)\|
=c\|\lozenge_{x^*}^{x^*}F(x)-\lozenge_\cx^\hx F(x)\|\notag\\
&\le c\ga_F \left(\max(\|\cx-x^*\|,\|\hx-x^*\|)\|x-x^*\|)+\tfrac12|\hx-x^*\|\,\|\cx-x^*\|\right)\\
&\le \frac23 \left(\|x-x^*\| + \tfrac12 \max(\|\cx-x^*\|,\|\hx-x^*\|)\right)\notag
\end{align}
which implies $\|x-x^*\|\le \max(\|\cx-x^*\|,\|\hx-x^*\|)\le R$. Hence, the distance from this root $x$ to
$\cx$ or $\hx$ does not exceed $2R$, and any root of $\lozenge_\cx^\hx F$ outside
$B_{\tilde R}(x^*)$ has a distance greater $\tilde R-R$ from these basis points, 
so that for $R\le\frac13\tilde R$ the result of the Newton operator 
consists only of roots inside $B_R(x^*)$. Moreover, via 
\[
N(\cx,\hx)\subset B(x^*,\max(\|\cx-x^*\|,\|\hx-x^*\|))
\]
we obtain strict contractivity
\end{proof}
\begin{corollary}
	In the situation of Proposition \ref{prop:contractivity}, if $F$ is (globally) coherently oriented, i.e., if $\mathcal D=\R^n$, then \[R\ =\ \frac2{3c\ga_F} \,.\]
\end{corollary}

\begin{theorem}(Quadratic or super-linear convergence)
Let $F\in {\rm span}(\Phi_{\rm abs})$ and $x^*$ be an isolated root of $F$ 
in an open neighborhood ${\cal D}$. If $\lozenge_{x^*}F$ is coherently oriented 
on ${\cal D}$, then there exist a radius $R>0$ so that for any initial point(s) 
$x_0\,(,x_1)\in B(x^*,R)$ the Newton iteration
$$
x_{j+1}\in N(x_j)
$$
in tangent mode  converges quadratically resp.
$$
x_{j+1}\in N(x_j,x_{j-1})
$$
in secant mode converges with order $\frac{1+\sqrt5}2$ 
towards $x^*$.
\end{theorem}

\begin{proof} Using the same constants as in the previous proof,
inequality \eqref{eq:Newton:basic} can be also transformed into 
\begin{align*}
 \|x_{j+1}-x^*\|&\le\frac32c\ga_F \|\cx-x^*\|\,\|\hx-x^*\|\le \frac1R\|\cx-x^*\|\,\|\hx-x^*\|
\end{align*}
where $(\cx,\hx)=(x_j,x_j)$ in tangent mode and $=(c_j,c_{j-1})$ in secant mode. The claim about the 
order of convergence follows directly.
\end{proof}

\begin{proposition}(Sufficient condition for convergence)\label{prop:sufficient}
	Let $F\in {\rm span}(\Phi_{\rm abs})$ be metrically regular at $x^*$. Then $\lozenge_{x^*}F$ is open in some neighborhood of $x^*$.
\end{proposition}
\begin{proof}
	If $F$ is metrically regular at $x^*$, the limiting Jacobians at $x^*$ are coherently oriented \cite[Thm. 2.4]{Fusek2013OnMR}. But it was shown in \cite{griewank2013stable} that the linear parts of the selection functions of $\lozenge_{x^*}F$ which are active at its development point $x^*$ are a subset of the limiting Jacobians of $F$ at $x^*$. Hence,\,$\lozenge_{x^*}F$ is coherently oriented and thus open in some neighborhood of $x^*$. 
\end{proof}
\begin{remark}(Open Problem)\label{rem:open-problem}
The argumentation of the proof of Proposition \ref{prop:sufficient} especially implies that if $F$ is metrically regular at $x^*$, there exists a neighborhood of $x^*$ such that all piecewise linearizations developed therein are locally coherently oriented in some neighborhood about their respective reference points. 
It is not clear, however, whether the intersection of these neighborhoods again contains a nonempty neighborhood about $x^*$. In this case all piecewise linearizations developed sufficiently close to $x^*$ would (locally) have the same number of isolated roots whose movements could then be tracked unambiguously since their respective ranges of motion would be bounded by the Lipschitz constants in Proposition \ref{prop:approx}.
 It was already noted in Remark \ref{rem:regular-root} that in general arbitrarily small perturbations of the reference point(s) of a piecewise linearization may cause an even number of $\mathcal O(2^s)$ additional roots to arise, leading to a potentially hard computation of the Newton sequence.
	\end{remark}

\begin{proposition}(Necessary condition for open linearization)
	Let $F\in {\rm span}(\Phi_{\rm abs})$ and assume that $F(x^*)=0$. Then, if $\lozenge_{x^*}F$ is metrically regular in some neighborhood about $x^*$, $F$ is open in $x^*$.
\end{proposition}
The proof is modeled after the standard proof for the implicit function theorem.
\begin{proof}
	Consider the fixed-point operator
	\[
	T(x)=x+x^*-(\lozenge_{x^*}F)^{-1}(F(x)-y)
	\]
	where the inverse is constrained to a suitably small neighborhood $B(x^*,\bar R)$ of $x^*$.
	By metric regularity, we have
	\begin{align}\label{eq:open:contractivity}
	\|T(x)-x^*\|\ &=\ \|x-(\lozenge_{x^*}F)^{-1}(F(x)-y)\|  \\
	& \le\  c\|\lozenge_{x^*}F(x)-F(x)+y\|  \\ \nonumber
	&\le\  c\frac{\ga_F}{2}\|x-x^*\|+c\|y\|\,. \nonumber
	\end{align}
	
	Set $R=\min\left(\bar R,\frac1{c\ga_F}\right)$. Then, if $\|y\|\le \frac{R}{2c}$, the 
	map $T$ maps $B(x^*,R)$ into itself. By Brouwer's fixed-point theorem,
	$T$ has a fixed point in $B(x^*,R)$, that is, $F(x)=y$.
	Re-evaluating \eqref{eq:open:contractivity}
	and using $c\ga_FR\le 1$, for the latter solution we get
		\begin{align}\label{eq:implicit}
			2R\|x-x^*\|\le \|x-x^*\|^2+2Rc\|y\|\le R\|x-x^*\|+2Rc\|y\|
			\implies \|x-x^*\|\le 2c\|y\|\,.
		\end{align}
	Re-inserting into \eqref{eq:implicit} the tighter bound 
	\[
	\|x-x^*\|\le \frac{c\|y\|}{1-\frac cR\|y\|}
	\]
	we obtain $B(0,\frac{\delta}{2c})\subset B(x^*,\delta)$ for all sufficiently small $\delta$.	
\end{proof}

\subsection{Counterexample}\label{sec:counterexamples}

For $0<x<y^2$ the following example function has a singular Jacobian. In this case \ssn steps are not defined because this especially means that in some neighborhood of the (unique) root no generalized Jacobian contains a nonsingular element.  

	\begin{align}
		\underset{(x, y) \in \R^2}{\mathrm{solve}:}\quad
		\begin{bmatrix}
			0 \\
			0
		\end{bmatrix} \overset !=
		f(x, y) &= 
		\begin{bmatrix}
			x + (y^2 - x_+)_+ \\
			y
		\end{bmatrix} = 
		\begin{cases}
			(x + y^2, y)^\top &\text{if } x \le 0 \\
			(y^2, y)^\top &\text{if } 0 < x \le y^2 \\
			(x, y)^\top &\text{else}
		\end{cases} \\
		\partial f(x, y) &= \begin{cases}
			\begin{bmatrix} 1 & 2y \\ 0 & 1 \end{bmatrix} &\text{if } x < 0 \\
			\begin{bmatrix} 0 & 2y \\ 0 & 1 \end{bmatrix} &\text{if } 0 < x < y^2 \\
			I_{2\times 2} &\text{if } y^2 < x 
		\end{cases}
		\label{eqn:KummerGriewankExample}
	\end{align}
However, one can easily check that the tangent mode piecewise linearization at the origin, which is also the root, is simply the identity. 
This is a surprising fact, given that singular matrices are contained within the set of all limiting Jacobians of \eqref{eqn:KummerGriewankExample} and that the set of all limiting Jacobians of its piecewise linearization at the root is a subset of the former. 
The reason is that all kinks in this example are linearized during the piecewise linearization process. The linerization of the parabola \(\{ (x, x^2) \mid x \in \R \}\) at the origin \((0, 0)\) equals the other kink \(\{(x, 0) \mid x \in \R\}\) which itself is kept invariant. 
Thus the union of the nonlinear sets \( \{ (x, y) \mid 0 < x < y^2 \} \) for \(y < 0\) and \(y > 0\), on which the limiting Jacobians are singular, is mapped onto empty polyhedra. 
This removes the singular matrices from the set of all limiting Jacobians of the piecewise linearization at the root, which is then trivially coherently oriented and so both generalized Newton's methods converge. 


\begin{figure}[htp]
	\centering
	\subfigure[First component of the non-linear function ]{\includegraphics[scale=.5]{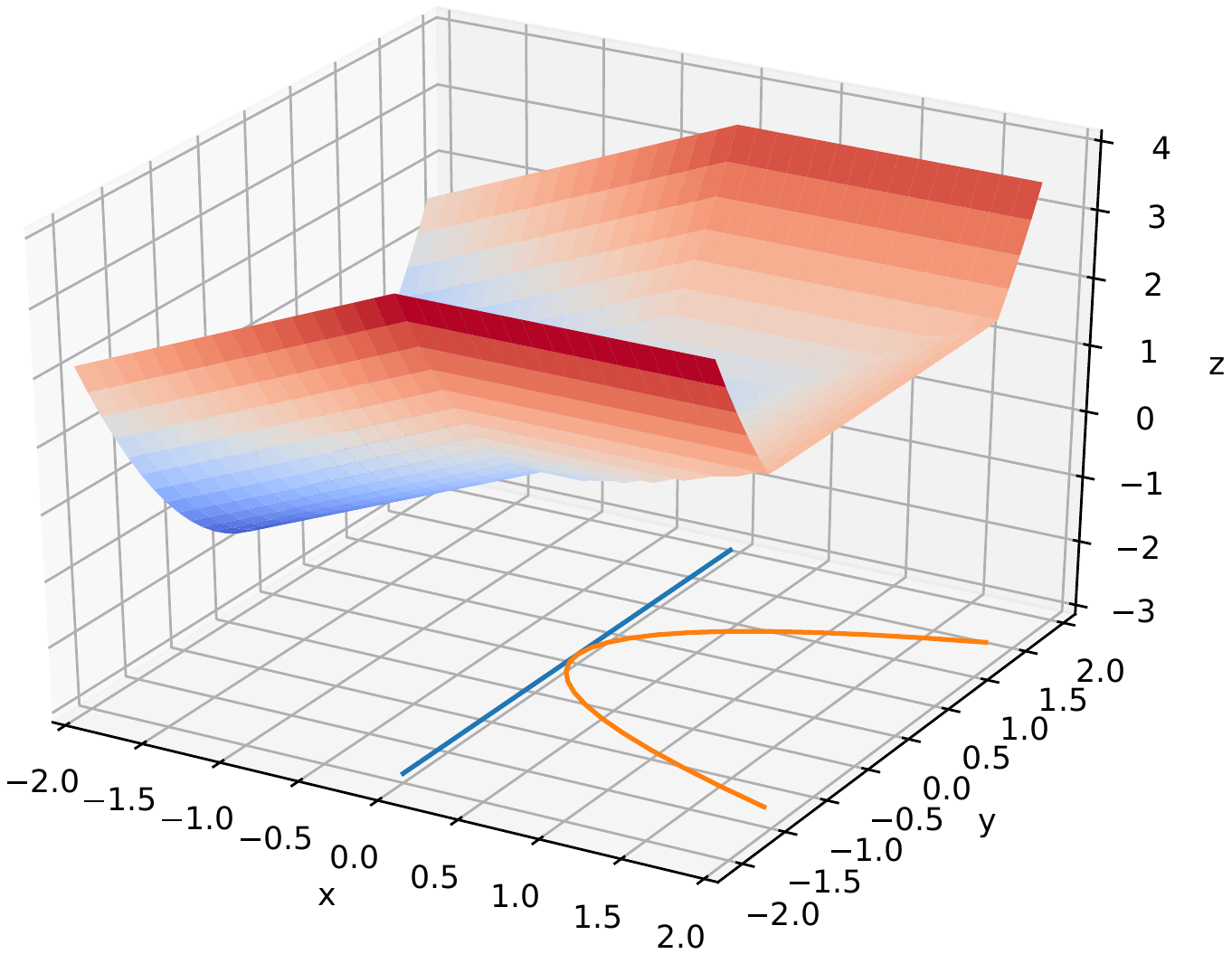}}
	\hspace{.2cm}
	\subfigure[First component of the PL approximation]{\includegraphics[scale=.59]{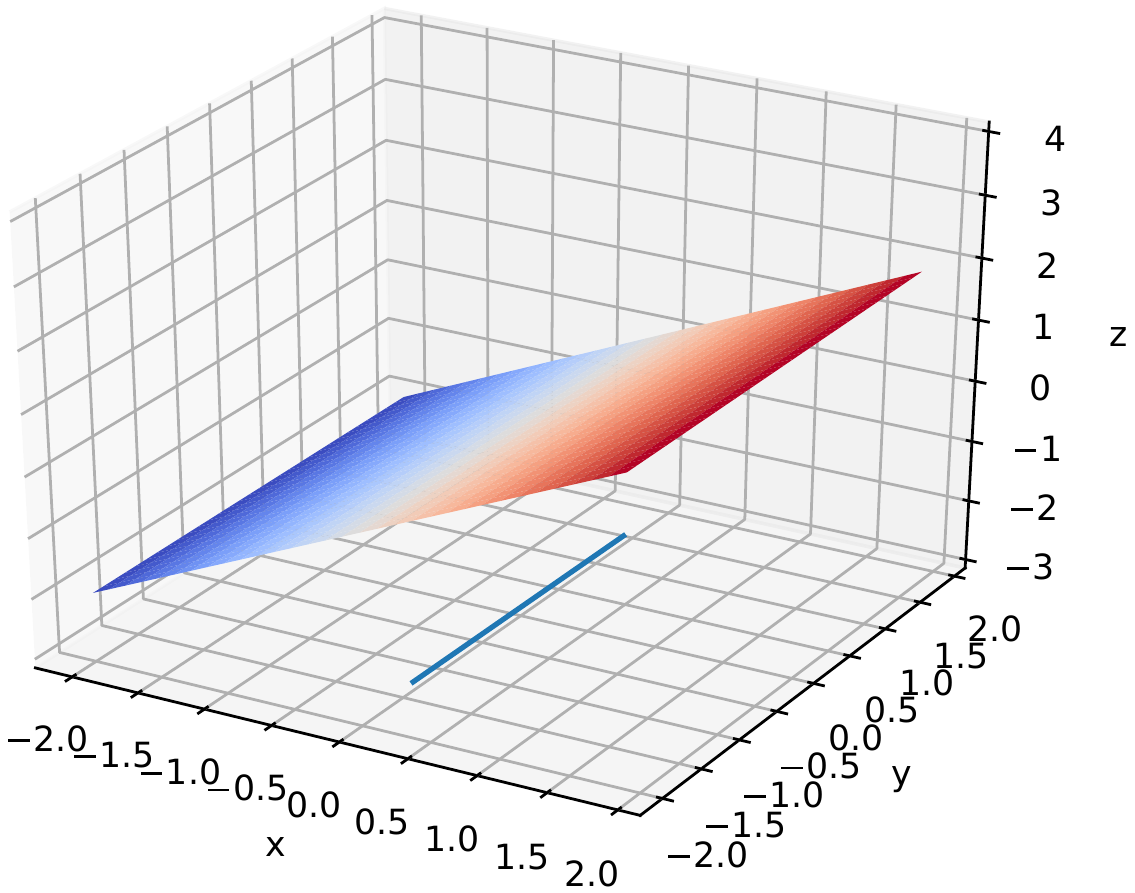}}
	\caption{Kink-Structure of \ssn Counterexample and its \pl approximation}	
	\label{Fig:SSN-counter}
\end{figure}

%

\section{Unambiguous Computation}\label{sec:unambiguous}

While the the local degree about a regular value $y$ of a continuous function $f:\R^n\rightarrow\R^n$ is constant under small perturbations of $f$, the number of solutions to the equation $f(x)=y$ may increase by any even number; cf Figure \ref{fig:pert} and Remarks \ref{rem:regular-root} and \ref{rem:open-problem}. In our context this means that during every Newton step one possibly faces the challenge of solving a piecewise linear system with exponentially many roots in a small ball about the last development point. 

However, we settled on a specific representation of our piecewise linearizations, the ANF. The dimensions of the ANF are predetermined by the structure of the evaluation graph of the underlying \pcs function; cf. Remark \ref{rem:ANF-dim}. 
 But while any \pl function $F:\Rm\rightarrow\Rn$ can be represented by an ANF, perturbations of the latter's entries -- while keeping its dimensions fixed -- cannot generate arbitrary perturbations of $F$. For example, we can represent the scalar function $F(x)=x$, which is trivially \pl, by an \anf where $J=1$ and all other blocks have dimension zero. Now $F$, as well as any function resulting from a slight perturbation of the entries of its \anf, i.e., of $J$, is a bijection. But not every \pl perturbation of $F$ is bijective; again cf. Figure \ref{fig:pert}. This lends a certain robustness to our Newton methods which we will now investigate further.

\begin{figure}
\centering
\subfigure{
\begin{tikzpicture}
\draw (-.5,0) ellipse (.5 and 1.25);
\draw (-.5,-1.25) -- (3.5,-1.25);
\draw (3.5,-1.25) arc (90:-90: .5 and -1.25);
\draw [dashed] (3.5,-1.25) arc (90:-90:-.5 and -1.25);
\draw (3.5,1.25) -- (-.5,1.25); 
\fill [gray,opacity=0.4] (-.5,-1.25) -- (3.5,-1.25) arc (90:-90:.5 and
-1.25) -- (-.5,1.25) arc (90:-90:-.5 and 1.25);
\fill[gray,opacity=0.1] (-.5,0) ellipse (.5 and 1.25);
\fill[gray,opacity=0.1] (3.5,0) ellipse (.5 and 1.25);
\draw[fill=black] (-.5, .5) circle (0.15em);
\draw[fill=black] (3.5, .5) circle (0.15em);
\draw[fill=black] (3.5, -.3) circle (0.15em);
\draw[fill=black] (3.5, -.8) circle (0.15em);

\draw[thick] (-.5,0.5) .. controls (.5,-.2) and (1.7,1) .. (3.5,.5);
\draw[thick] (3.5,-.3) .. controls (3,-.8) and (1.8,0) .. (1.4,-.55);
\draw[thick] (1.4,-.55) .. controls (1.,-1.2) and (3,-.5) .. (3.5,-.8);

\draw[fill=black] (.7, .31) circle (0.15em);
\draw[fill=black] (2., .58) circle (0.15em);
\draw[fill=black] (2., -.36) circle (0.15em);
\draw[fill=black] (2., -.81) circle (0.15em);

\fill[gray,opacity=0.3] (.7,0) ellipse (.5 and 1.25);
\draw (.7,-1.25) arc (90:-90: .5 and -1.25);
\draw [dashed] (.7,-1.25) arc (90:-90:-.5 and -1.25);
\draw (.1,1.25) -- (0.1,1.7) -- (1.3,1.5) -- (1.3,1.5) -- (1.3,-1.7) --
(0.1,-1.5) -- (.1,-1.25);
\draw[dashed] (.1,-1.25) -- (.1,1.25);

\fill[gray,opacity=0.3] (2.,0) ellipse (.5 and 1.25);
\draw (2.,-1.25) arc (90:-90: .5 and -1.25);
\draw [dashed] (2.,-1.25) arc (90:-90:-.5 and -1.25);
\draw (1.4,1.25) -- (1.4,1.7) -- (2.7,1.5) -- (2.7,1.5) -- (2.7,-1.7) --
(1.4,-1.5) -- (1.4,-1.25);
\draw[dashed] (1.4,-1.25) -- (1.4,1.25);
\end{tikzpicture}
}
\hspace*{.7cm}
\subfigure{
\begin{tikzpicture}
\begin{axis}[xmin=-3,xmax=9,yscale=.5,ticks=none,axis line
style={draw=none}]
\addplot[color=black,line width=1pt,domain=-2:1,samples=500]
    {.75*(x+2)};
\addplot[color=black,line width=1pt,domain=4:8,samples=500]
    {min(0.75*(x-4),max(0.75*(-x+7.43),0.75*(x-5)))};
\addplot[color=lightgray,line width=1pt,domain=-3:2,samples=500]
    {1.1};
\addplot[color=lightgray,line width=1pt,domain=3:9,samples=500]
    {1.1};
\end{axis}
\draw[fill=black] (1.4, 1.4) circle (0.15em);
\draw[fill=black] (4.83,1.4) circle (0.15em);
\draw[fill=black] (5.12,1.4) circle (0.15em);
\draw[fill=black] (5.4, 1.4) circle (0.15em);
\end{tikzpicture}
}
\caption{Solution increase by perturbation}
\label{fig:pert}
\end{figure}
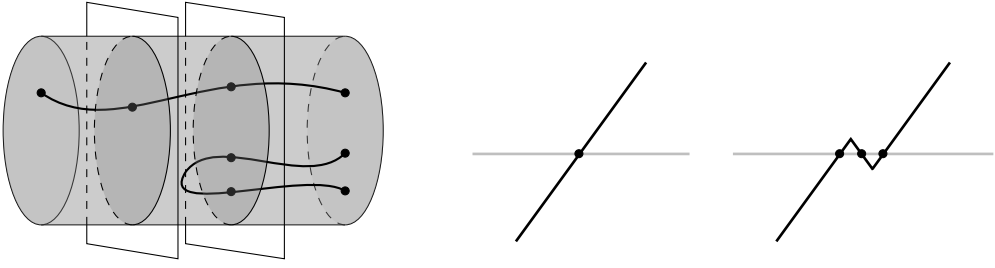

\subsection{Abs-Normal Form and Absolute Value Equation}\label{sec:ANF-perturbations}
In \cite{griewank2014abs} it was shown that $J$, the smooth part of the abs-normal form, can be assumed, without loss of generality, to be nonsingular. Moreover, by fixing $y$ and subsuming it into $b$, for regular $J$ the transformation 
\begin{align}\label{AVE}
 z  - L | z|   + Z J^{-1} Y |z| \; \equiv \; z-S\vert z\vert  \; = \;  \hat c \; \equiv \;   c  - ZJ^{-1}b  
\end{align}
of the abs-normal form into a so-called absolute value equation (AVE) was derived and the 1-1-solution correspondence of both systems proved. Denote by $\sig$ the set of $(s\times s)$- signature matrices, that is, diagonal matrices with entries in $\{-1,1\}$. Then the limiting Jacobians of the piecewise linear function 
\[F_S:\Rs\rightarrow\Rs\,,\qquad z\mapsto z-S\vert z\vert,\] whose facets are the orthants of $\Rs$, are the matrices $I-S\Sigma$, where $\Sigma\in\sig$. It is well known that $F_S$ is bijective if and only if all its limiting Jacobians have a positive determinant sign, cf. \cite{rump1997theorems}. Due to the continuity of the determinant, $F$ is bijective if and only if it is stably bijective in the sense that it is bijective under sufficiently small perturbations of $I$ and $S$. We remark that the determinants of the $(I-S\Sigma)$ cannot all be negative, because by the linearity of the determinant with respect to rank-$1$ updates all matrices within their convex hull would then have a negative determinant -- including the identity.

Let $F:\Rn\to\Rm$ be the piecewise linear function in \anf from which $F_S$ was derived. 
For $x\in\Rn$ and the associated vector of intermediates $z(x)$ we define $\Sigma(x)$, the signature at $x$, as the signature matrix in $\sig$ which satisfies $z(x)\Sigma(x)=\vert z(x)\vert$. 
It was shown in \cite{griewank2013stable} that the sets $P_\Sigma\equiv\{x\in\Rn:\Sigma(x)=\Sigma \}$ are relatively open disjoint polyhedra. Moreover, the limiting Jacobians of $F$ are the linear parts of the selection functions which are active on some $P_\Sigma$ which justifies denoting them by $J_\Sigma$.  
By \cite[Lem. 6.1]{griewank2014abs}, we have
\begin{align}\label{eq:ANF-det}
	\det(J_\Sigma)\ =\ \det(J)\cdot\det(I-S\Sigma)\,.
\end{align}
However, while $F_S$ always has $2^s$ limiting Jacobians $I-S\Sigma$, each one corresponding 
to the interior of an orthant of $\R^s$, some $P_\Sigma$ may be empty. In this case $F$ does 
not have a limiting Jacobian $J_\Sigma$. The case that $F$ has $2^s$ limiting Jacobians $J_\Sigma$ is called \textit{totally switched}. It occurs if and only if $Z$ is surjective. This requires
$s\leq n$, cf. \cite{griewank2014abs}. Now, if a \pl function in ANF is coherently oriented \textit{and} totally switched, equality \eqref{eq:ANF-det} implies that all matrices of the form $I-S\Sigma$, where $\Sigma\in\sig$, have the same nonzero determinant sign which then has to be positive. Due to the 1-1-correspondence of ANF and associated AVE solutions, this proves 
\begin{lemma}\label{lem:totally-switched}
	Let $F:\Rn\rightarrow\Rn$ be a \pl function in \anf which is coherently oriented and totally switched. Then $F$ is bijective.
\end{lemma}

\subsection{Stable Coherent Orientation}
Recall that a \pl function is called coherently oriented if the linear parts of its selection functions all have the same nonzero determinant sign \cite{scholtespl}. Piecewise linear functions are surjective if they are coherently oriented and open if and only if they are coherently oriented. Due to the nearness property of the mapping degree we have 
$\operatorname{deg}(f, \Omega,y)=\operatorname{deg}(f, \Omega,y')$ for any two regular values of a surjective continuous function $f:\Rn\rightarrow\Rn$, where $\Omega$ is an open region that contains the preimages of both $y$ and $y'$. It is then justified to speak of \textit{the degree} of $f$.

Since the determinant signs of the differentials of its selection functions are either all positive or all negative, a coherently oriented piecewise linear function $F:\Rn\to\Rn$ has nonzero degree and all regular values have the same number of preimages, which equals the degree of $F$. For degree $>1$ or $<-1$ this implies that $F$ is a branched covering. For degree $\pm 1$ we have the following
\begin{lemma}\label{lem:degree-one}
	A piecewise linear function $F:\Rn\rightarrow\Rn$ is a homeomorphism if and only if it is coherently oriented of degree $\pm 1$.
\end{lemma}
\begin{proof}
	If $F$ is homeomorphic and thus bijective, it has a globally defined degree which has to be $\pm 1$ since each of its regular values has precisely one preimage. But then the differentials of the 
	selection functions of $F$ either all have determinant sign $+1$ or all have determinant sign $-1$. 
	
	Now assume $F$ is coherently oriented of degree $\pm 1$. Then every regular value has exactly one preimage. All critical values have at least one preimage due to the surjectivity of $F$ which is implied by its coherent orientation. No critical value of $F$ can have an infinite number of preimages since all selection functions of $F$ are regular. Now assume there exists a critical value $y$ with a finite number of preimages, say $F^{-1}(y)=\{x_1, x_2,\dots,x_k\}$, where $k>1$. Let $U_1, U_2,\dots,U_k$ be open neighborhoods of $x_1, x_2, \dots, x_k$. Then $\cap_{i\in[k]}F(U_i)$ is a nonempty open set due to the openness of $F$ and thus contains a regular value, which then has $k$ preimages by construction. But this contradicts the hypothesis and proves bijectivity of $F$. By its openness it follows that $F$ is a homeomorphism.
\end{proof}

We call a \pl function in \anf \textit{stably coherently oriented} if all modifications generated by small perturbations of $Z$ are also coherently oriented. We define \textit{stable bijectivity} analogously.
\begin{theorem}\label{thm:stable-bijectivity}
	Let $F:\Rn\rightarrow\Rn$ be a piecewise linear function in \anf representation with nonsingular block $J$ and $s\leq n$. Then the following are equivalent: 
	\begin{enumerate}
		\item $F$ is stably coherently oriented,
		\item $F$ is stably bijective,
		\item $F$ is coherently oriented under small perturbations of $Z,L,J, Y, c, b$,
		\item $F$ is bijective under small perturbations of $Z,L,J, Y,c,b$,
		\item the AVE associated to $F$ is uniquely solvable.
	\end{enumerate}
\end{theorem} 
\begin{proof}
	The entries of $S$ depend continuously on the entries of the ANF. Hence, the implications "$5.\Rightarrow 1., 2., 3., 4.$" are proved by the aforementioned 1-1-correspondence between solutions of \anf and associated AVE. Moreover, $2., 3., 4.$ imply $1.$ since a bijective \pl function is coherently oriented, cf. \cite{scholtespl}. Thus showing "$1.\Rightarrow 5.$" completes the proof. 

	Let $U$ be a neighborhood of $Z$ in $\R^{s\times n}$ such that for all $\tilde Z\in U$ the corresponding perturbation $\tilde F$ of $F$ is coherently oriented. Further, let $B\subset U$ be an open ball about $Z$. Then there exists some full rank $\tilde Z\in B$ and the corresponding perturbation $\tilde F$ is totally switched. By Lemma \ref{lem:totally-switched} $\tilde F$ is bijective. Now all matrices $Z_t:=t\tilde{Z}+(1-t)Z$, where $t\in[0,1]$, lie in $B$, so that the corresponding perturbations $F_t$ of $F$ are coherently oriented. Since the critical values of $F$ and $\tilde F$ are restricted to the images of the $(n-1)$-skeleta of their respective polyhedral domain decompositions, which have measure zero in the range, we can find a $y\in\Rn$ which is regular for both $F$ and $\tilde F$. Since the selection functions of all $F_t$ are affine isomorphisms, $F_t^{-1}(y)$ is compact for all $t\in[0,1]$. Hence, there exist some bounded region $\Omega\subset\Rn$ such that $F_{[0,1]}^{-1}(y)\subset\Omega\times[0,1]$. But this implies $\operatorname{deg}(F,\Omega,y)=\operatorname{deg}(\tilde F,\Omega,y)$. Now consider Lemma \ref{lem:degree-one}.  
\end{proof}

Using Theorem \ref{thm:stable-bijectivity}, we get from Proposition \ref{prop:approx}.2:
\begin{theorem}\label{thm:stable-bij-convergence}
Let \(F\in\text{span}(\Phi_\abs)\), \(x^\ast\) a root of \(F\), and assume \(\lozenge_{x^\ast}F\) satisfies any of the equivalent conditions in Theorem \ref{thm:stable-bijectivity}. Then there exists a ball $B(x^*, \rho)$ such that all tangent and secant linearizations with development points in $B(x^*, \rho)$ are bijective and both the tangent and secant mode piecewise differentiable Newton's methods converge from all starting points that lie in $B_R[x^*]$, where 
\[R\ :=\ \frac 13\, \min\left(\rho, \frac{1}{2c\gamma_F}\right)\; .\]
Moreover, for both methods the Newton path is uniquely determined by the respective starting point of the iteration.
\end{theorem}
In the situation of Theorem \ref{thm:stable-bij-convergence} the Newton steps can be calculated 
in (weakly) polynomial time via interior point methods for LCPs \cite{potra2007interiorpoint}, using the
equivalence of AVE and \lcp. Algorithms which are more efficient for special system structures
can be found, e.g., in \cite{brugnano2009iterative}, \cite{griewank2014abs}, \cite{SGE}.  

\subsection{Unstable Bijectivity} 

For every \(c \ge 0\) the following ANF is equivalent to the identity of \(x\), but it is not injective otherwise:
\begin{align*}
\begin{bmatrix}
z_0 \\
z_1 \\
y
\end{bmatrix} = 
\begin{bmatrix}
-c \\
0 \\
0
\end{bmatrix}+
\left[\begin{array}{c|cc}
1 & 0 & 0 \\
1 & 1 & 0 \\
\hline
\tfrac 12 & -\tfrac 12 & \tfrac 12
\end{array}\right] \cdot
\begin{bmatrix}
x \\
|z_0| \\
|z_1|
\end{bmatrix}
\end{align*}
We have
\begin{align*}
S\ =\ L - ZJ^{-1}Y\ =\
\begin{bmatrix}
0 & 0 \\
1 & 0
\end{bmatrix} - 
\begin{bmatrix}
-1 & 1 \\
-1 & 1
\end{bmatrix} =\ 
\begin{bmatrix}
1 & -1 \\
2 & -1
\end{bmatrix}\,.
\end{align*}
One can easily check that the matrices $I-S\Sigma$ are not coherently oriented, which implies that the corresponding \pl function, say $F_S$, is not bijective.
This is not a contradiction to the aforementioned 1-1-solution correspondence between \anf and AVE. Since $s>n$, the \anf solutions map into some lower dimensional set $U\subset\Rs$, which translates to local bijectivity of $F_S$ on $U$.

\section{Cardiovascular System}\label{sec:numerics}

In this section we want to demonstrate the tangent mode generalized Newton method applied to a series of nonlinear and nonsmooth systems of equations as they arise from solving differential algebraic equations (DAE) numerically.
The numerical instance in this experiment is a modification of a so-called lumped parameter model of the \emph{human cardiovascular system} introduced in \cite{Ottesen_cava}.
From a modeling point of view such systems are segmented into compartments, which in general wrap different parts of the circulatory system. 
In our case there are \(14\) of them. 
\(4\) compartments represent all the heart chambers. 
These are the left and right atrium (la and ra) or pre-chamber as well as the left and right ventricle (lv and rv) also referred to as main chambers. 
Finally, each of the systemic and pulmonary circulations are represented by \(5\) compartments, whose blood vessels can be subdivided into \(5\) different categories:
\begin{description}
  \item[arterial system] larger arteries (sa1 resp. pa1),\; 
  arteries and arterioles (sa2 resp. pa2)
  \item[capillaries] arterioles and capillaries (sa3 resp. pa3)
  \item[vein system] veins and venules (sv2 resp. pv2),\; larger veins (sv1 resp. pv1)
\end{description}
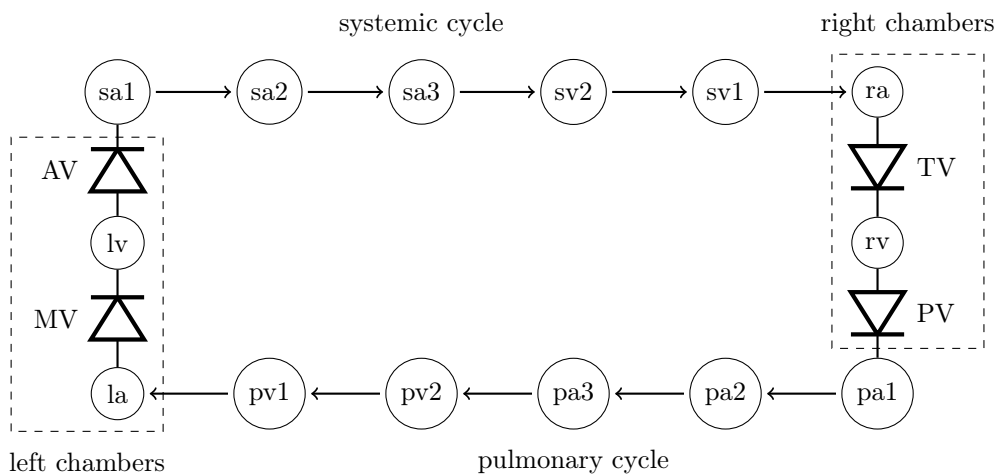
\begin{figure}[ht]\centering
  \tikzset{
      pil/.style={
             ->,
             thick,
             shorten <=2pt,
             shorten >=2pt,}
  }
  \begin{tikzpicture}
      \node[draw, circle] (la)  at (0, 0)  {la};
      \node[draw, circle] (lv)  at (0, 2)  {lv};
      \node[draw, circle] (sa1) at (0, 4)  {sa1};
      \node[draw, circle] (sa2) at (2, 4)  {sa2};
      \node[draw, circle] (sa3) at (4, 4)  {sa3};
      \node[draw, circle] (sv2) at (6, 4)  {sv2};
      \node[draw, circle] (sv1) at (8, 4)  {sv1};
      \node[draw, circle] (ra)  at (10, 4) {ra};
      \node[draw, circle] (rv)  at (10, 2) {rv};
      \node[draw, circle] (pa1) at (10, 0) {pa1};
      \node[draw, circle] (pa2) at (8, 0)  {pa2};
      \node[draw, circle] (pa3) at (6, 0)  {pa3};
      \node[draw, circle] (pv2) at (4, 0)  {pv2};
      \node[draw, circle] (pv1) at (2, 0)  {pv1};

      \path[every node/.style={font=\sffamily\small}]
       (sa1) edge[pil] (sa2)
       (sa2) edge[pil] (sa3)
       (sa3) edge[pil] (sv2)
       (sv2) edge[pil] (sv1)
       (sv1) edge[pil] (ra)
       (pa1) edge[pil] (pa2)
       (pa2) edge[pil] (pa3)
       (pa3) edge[pil] (pv2)
       (pv2) edge[pil] (pv1)
       (pv1) edge[pil] (la);

      \draw[color=black, thick] (la.north) to [Do, l=MV] (lv.south);
      \draw[color=black, thick] (lv.north) to [Do, l=AV] (sa1.south);
      \draw[color=black, thick] (ra.south) to [Do, l=TV] (rv.north);
      \draw[color=black, thick] (rv.south) to [Do, l=PV] (pa1.north);

      \draw[dashed] (-1.4, -0.5) rectangle (0.6, 3.4);
      \draw[dashed] (9.4, 4.5) rectangle (11.4, 0.6);

      \node at (-0.4, -0.9) (left) {left chambers};
      \node at (10.4, 4.9) (right) {right chambers};
      \node at (4, 4.9) (sys) {systemic cycle};
      \node at (6, -0.9) (oul) {pulmonary cycle};
  \end{tikzpicture}
  \caption{Schematic overview of all \(14\) compartments and their relations.}
  \label{pic:cava}
\end{figure}
The diodes in figure \ref{pic:cava} (MV - Mitral valve, AV - Aortic valve, TV - Tricuspid valve and PV - Pulmonary valve) indicate positions of the \(4\) heart valves, which allow a one-way low drag directional flow. 
%
%
%
In contrast to the original source we replace the discontinuous or binary behavior of the valves by a piecewise linear model resembling the anatomic process slightly more accurately.
Finally for the modified system we introduce the following equations for its compartments:

\vspace{.2cm}
\begin{center}
\begin{tabular}{cc}
Large Vessels \(k \in \left\{\begin{gathered} \rm{sa1}, \rm{sv1}, \\ \rm{pa1}, \rm{pv1}\end{gathered}\right\}\) &
\qquad\qquad smaller vessels \(k \in \left\{\begin{gathered}\rm{sa2}, \rm{sa3}, \rm{sv2}, \\ \rm{pa2}, \rm{pa3}, \rm{pv2}\end{gathered}\right\}\) \qquad\qquad \\ & \\
\(\begin{aligned}
	L\cdot\dot q_k &= p_k - p_\suc - R_k\cdot q_k, \\
	\dot v_k &= q_\pre - q_k, \\
	C_k\cdot p_k &= v_k - V_k,
\end{aligned}\) &
\(\begin{aligned}
	R_k\cdot q_k &= p_k - p_\suc, \\
	\dot v_k &= q_\pre - q_k, \\
	C_k\cdot p_k &= v_k - V_k,
\end{aligned}\)
\end{tabular}
\end{center}

\vspace{.2cm}

where \(p_\suc\) means the pressure within the succeeding compartment and \(q_\pre\) means the blood flow through the preceding compartment, in the sense of figure \ref{pic:cava}.
In advance \(L_k, R_k, C_k, V_k\) are compartment-specific, positive constants. 
The heart chambers are modeled as follows: 
\begin{center}
\begin{tabular}{cc}
Atria \(k \in \{\rm{la}, \rm{ra}\}\) &
Ventricles \(k \in \{\rm{lv}, \rm{rv}\}\) \\ & \\
\(\begin{aligned}
	\theta_k &= \tfrac 12(|p_k - p_\suc| \\
	     &\qquad - |p_k - p_\suc - 1| + 1), \\
	(1-\theta_k)q_k &= \theta_k(p_k - p_\suc - R_k q_k - L_k \dot q_k), \\
	\dot v_k &= q_\pre - q_k, \\
	p_k &= E_k\cdot (v_k - V_k),
\end{aligned}\) &
\(\begin{aligned}
	\theta_k - \tfrac 12 &= \tfrac 12(|p_k - p_\suc - R_k q_k| \\
	     &\qquad - |p_k - p_\suc - R_k q_k - 1|), \\
	(1-\theta_k)q_k &= \theta_k(p_k - p_\suc - R_k q_k - L_k \dot q_k), \\
	\dot v_k &= q_\pre - q_k, \\
	p_k &= E_k(t)\cdot (v_k - V_k),
\end{aligned}\)
\end{tabular}
\end{center}
where \(E_{\rm{la}}, E_{\rm{ra}}\) are more positive constants for the atria, whereas \(E_k(t)\) are so called elastance functions for the ventricles \(k \in \{\rm{lv}, \rm{rv}\}\) and defined by:
\begin{align*}
  && E_k(t) &= E_{\min{}\text{, k}}(1-\phi(t)) + E_{\max{}\text{, k}}\phi(t) \\
  \text{where}&&\phi(t) &= \max\left(0, \alpha\sin\left(\bar t\right) - \beta\sin\left(2\bar t\right)\right) \quad\text{ and }\quad 
  \bar t = \pi\frac{t\!\!\!\!\mod t_h}{\kappa_0 + \kappa_1t_h},
\end{align*}
with \(E_{\min{}\text{, k}}, E_{\max{}\text{, k}}\), as well as \(\alpha, \beta, t_h, \kappa_i\), our last positive constants. Representative values for all constants can be found in the aforementioned reference \cite{Ottesen_cava}. Note that we made use of the following identity for the cut-off function in our modeling:
\begin{align*}
	\theta_k \equiv \max(0, \min(1, x)) = \tfrac 12(|x| - |x - 1| + 1).
\end{align*}

Now we focus on solving the iterative sequence of nonsmooth and nonlinear systems of equations:
\begin{align}
        \forall i \in \{0, 1, \ldots, N\}:\; \underset{x_i}{\mathrm{solve}}\;\, 0\; =\; F_i(x_i) \equiv \mathfrak f\left( \tfrac{\alpha_i}hD\cdot x_i + \tfrac 1h D\cdot x_{i-1},\; x_i,\; t_0 + i\cdot h \right). \label{eqn:itSeries}
\end{align}
which arise from time discretization, by the implicit Euler method, applied to the differential algebraic description of the human cardiovascular system:
\begin{align}
        0 = \mathfrak f\left(\tfrac{\mathrm d}{\mathrm d t} \big[ D\cdot x(t) \big],\, x(t), t\right),\quad \text{with}\quad
        D = \diag(\delta_{\mathbf v}, \delta_{\mathbf q}, \delta_{\mathbf p}), \label{eqn:CavaDAE}
\end{align}
where \(x = (\mathbf v, \mathbf q, \mathbf p)\) is a vector composed of all intrinsic variables, with \(\mathbf v = (v_k)_{k \in \mathcal C}\) the blood volumes in each compartment \(k \in \mathcal C\), \(\mathbf q = (q_k)_{k \in \mathcal C}\) the blood flows through all compartments and \(\mathbf p = (p_k)_{k \in \mathcal C}\) the blood pressures within all compartments. Furthermore, we have \(\delta_{\mathbf v} = (1, \dots, 1)\), \(\delta_{\mathbf p} = (0, \dots, 0)\) and
\begin{align*} 
  \delta_{\mathbf q, k} = \begin{cases} 
    0 &\text{if } k \text{ is a smaller vessel} \\ 
    1 &\text{else} 
  \end{cases}.
\end{align*} 
Thus \(\mathcal C = \{ \rm{la}, \rm{lv}, \rm{sa1}, \rm{sa2}, \rm{sa3}, \rm{sv2}, \rm{sv1}, \rm{ra}, \rm{rv}, \rm{pa1}, \rm{pa2}, \rm{pa3}, \rm{pv2}, \rm{pv1} \}\) make up the set of all compartments of our system.
The units of the \(3\) intrinsic variables are mmHg or millimeter of mercury for pressures, ml/s or milliliters per second for flows and ml or milliliters for volumes.

\begin{figure}[htp]\centering
        \includegraphics[width=\linewidth]{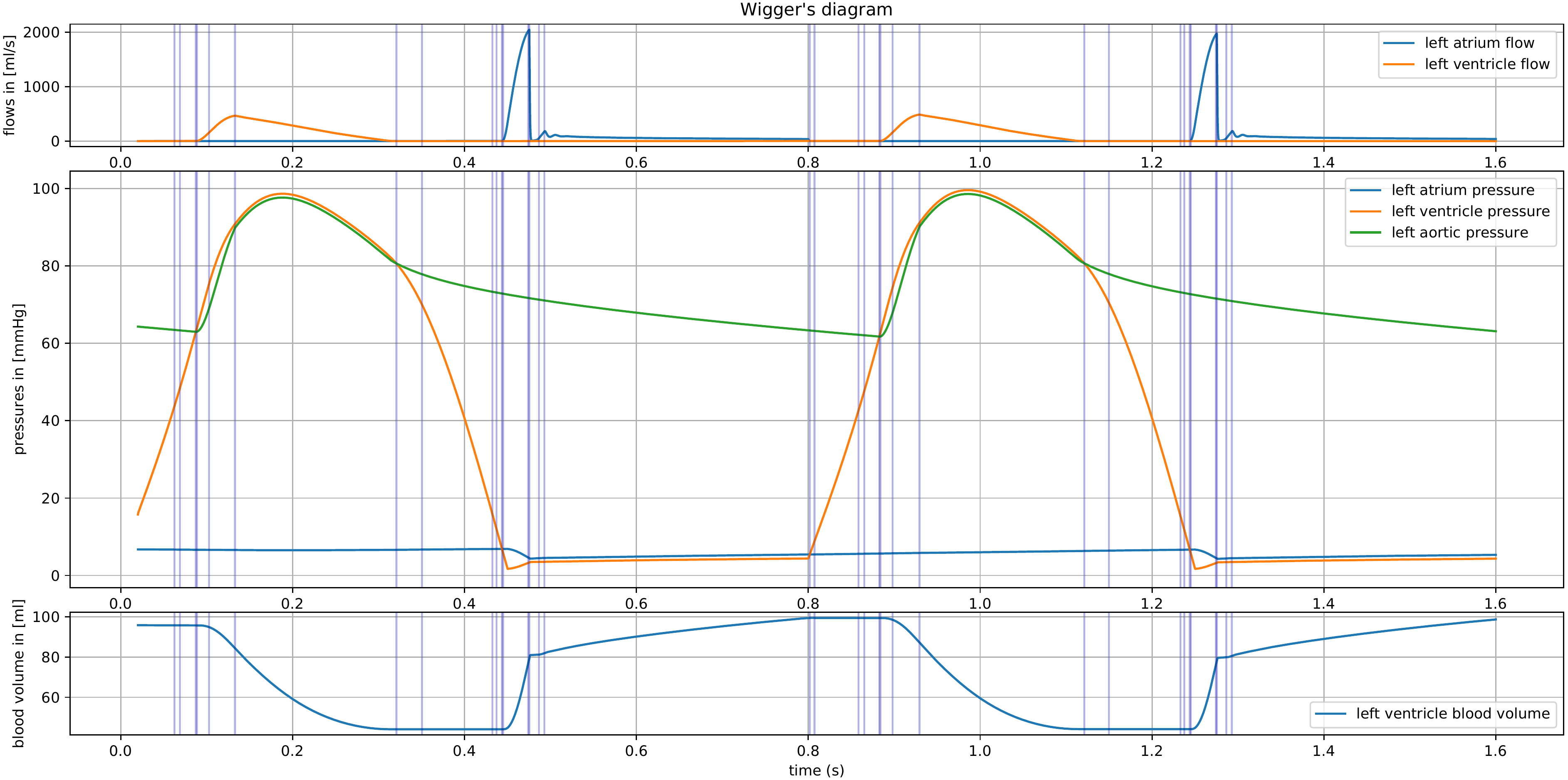}
  \caption{Numerical Wigger's diagram for left heart chambers. Blue background bars define intervals of length \(h\), where kinks were crossed to solve the piecewise linearizations.}
  \label{pic:leftWiggers}
  
        \includegraphics[width=\linewidth]{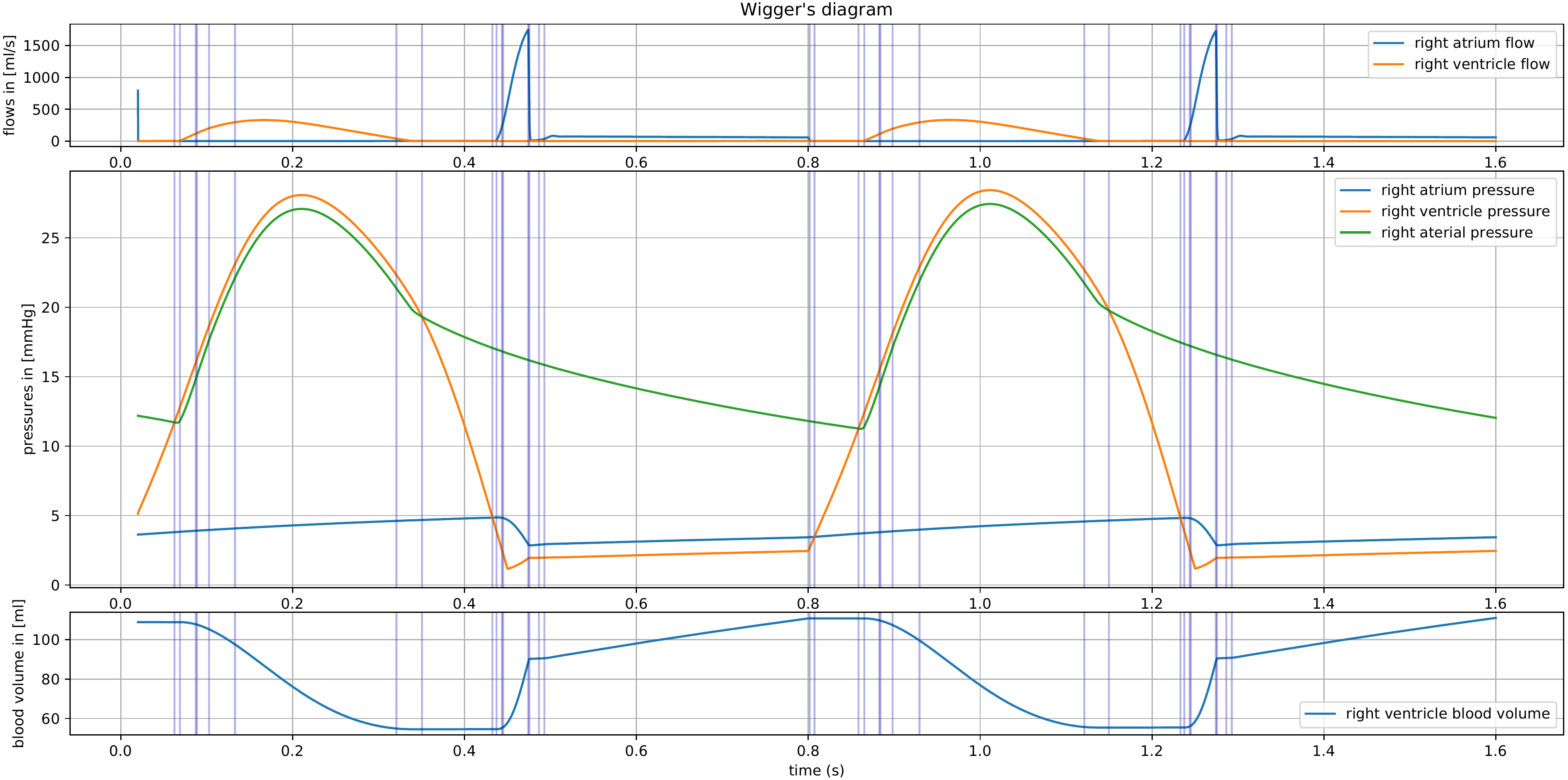}
  \caption{Numerical Wigger's diagram for right heart chambers. Blue background bars define intervals of length \(h\), where kinks were crossed to solve the piecewise linearizations.}
  \label{pic:rightWiggers}
\end{figure}

\begin{figure}[ht]\centering
        \includegraphics[width=\linewidth]{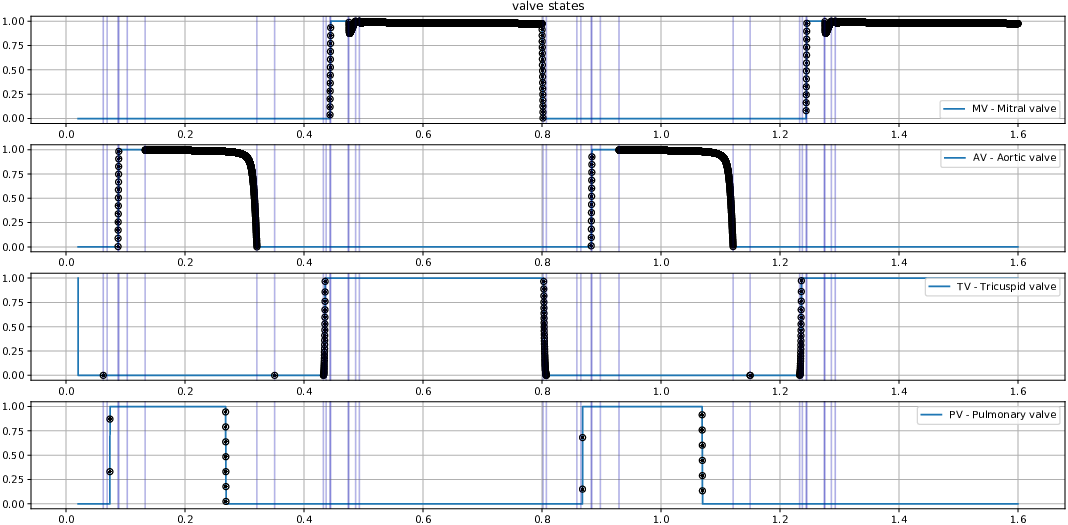}
  \caption{State of all \(4\) valves. Blue background bars define intervals of length \(h\), where kinks were crossed to solve the piecewise linearizations. Black bullets are points of evaluation during the transition of the corresponding valve.}
  \label{pic:valveStates}
\end{figure}

We have solved the sequence of nonlinear equation systems \eqref{eqn:itSeries} for different time step values \(h = 1e{-}3, 1e{-}4, 1e{-}5\). 
Whenever each of the valves is either closed tightly or opened fully, the system is piecewise linear. 
Thus exactly one Newton step will already solve it.
Whereas during the transition of at least one of the valves, the system to solve becomes quadratic in \(1\) up to \(4\) of its components.
Here a minimal step size of at least \(h \le 1e{-}3\) is required because a sufficient number of evaluation points during the transition of all valves is necessary to avoid strange backflow behavior.
This is a consequence of using a non-generalized integrator for nonsmooth differential equations (for further details on generalized integrators see \cite{PDODE}).

Certain components of the numerical solution for $h=1e{-}4$ are depicted in Figures \ref{pic:leftWiggers} and \ref{pic:rightWiggers}.
These two figures resemble the so-called Wigger's diagram, as it can be found in many medical textbooks. 
The original diagram shows idealized curves of blood pressures in the pre-chamber or atrium, main chamber or ventricle and leaving blood vessels from the heart during one or two heart periods. 
It also displays the blood volume in the main chamber and other quantities which are not considered in our modeling.
We have included blood flows from the pre- into the main chamber and from the main chamber into leaving blood vessels, since they are affected directly by valve actions.
The blue bars in the background of both figures define intervals of length \(h\) where solving the piecewise linearization resulted in a change of the signature of switching variables or, in other words, where kinks have been crossed. 
In any case we have experienced one kink-crossing at most. Further, the tangent mode generalized Newton method solves almost all instances of \eqref{eqn:itSeries} in exactly one step regardless of the states of the valves (see Table \ref{tab:stats}).
\begin{table}[ht]\centering
  \begin{tabular}{|c||c|c|}
    \hline
    \(h\) & \(1\) Newton step & \(2\) Newton steps \\
    \hline
    \(1e{-}3\) & \(1580\) & \(0\) \\
    \(1e{-}4\) & \(15799\) & \(1\) \\
    \(1e{-}5\) & \(158000\) & \(0\) \\
    \hline
  \end{tabular}
  \caption{Number of equation solvings and Newton steps per stepsize.}
  \label{tab:stats}
\end{table} 
The solutions of all piecewise linearizations are contained in the interior of some polyhedron from its polyhedral decomposition.
The linear operator of the corresponding linear selection function which is active on the same polyhedron proved to be nonsingular. Thus local convergence for the generalized Newton method is guaranteed in the sense of this paper. Moreover, the local coherent orientation throughout all iterations guarantees unambiguously trackable roots in the sense of Remark \ref{rem:open-problem}. 

A \pl solver suitable to the purpose of solution tracking on successive perturbations of a \pl system, which stay locally coherently oriented about their respective reference points throughout, is the \pl Newton method described in \cite{griewank2013stable}. 
It was started with the signature of the reference point of the current piecewise linearization 
and never needed more than one iteration, which corresponds to the fact that the development points for the tangent mode piecewise linearizations never crossed more than one kink.

\section{Conclusion and Final Remarks}\label{sec:conclusion}

The focus of this work has been the analysis of the Newton type algorithms developed in \cite{NewtonPL}. Two types of local \pl approximations, tangent and secant mode, were plugged into the general algorithmic scheme of semismooth Newton. Mapping degree theory was used to show: 
if the tangent mode piecewise linearization at an isolated root of a \pcs function is locally 
coherently oriented, then the image of small perturbations of this local model still contains 
a ball about the origin, yielding convergence of our methods under this condition.

A key feature of the tangent mode piecewise linearizations is that the limiting Jacobians 
at their reference point $\rx$ coincide with a subset of the -- but not necessarily all -- 
limiting Jacobians of the underlying \pcs function at $\rx$. This leads to a certain robustness
of our method in that we may still have guaranteed convergence even if there exists no 
neighborhood of the root throughout which classical (semismooth) Newton steps are defined. 

The gain in robustness comes at the cost of limiting our applications to \pcs 
functions. However, in the context of actual implementations of functions, the condition 
of finite evaluation procedures does not seem like a severe limitation. The numerical example in 
Section \ref{sec:numerics} shows a real world example that falls into this class. 

Another aspect that necessitates careful consideration is the fact that solving piecewise 
linear system is a potentially hard problem. However, the phenomenological observation is
that the requirement of local openness of the \pcs function enforces enough "good" structure
on the piecewise linear model so that the Newton iterates can usually be computed at essentially
the cost of the solution of a linear system. In the case that the local model is stably 
coherently oriented at the root, a (weakly) polynomial cost is guaranteed in a sufficiently small neighborhood.

In the present article we proved a number of statements which could be qualified as "exact".
For future endeavors our interest lies in matters concerning perturbations. For example,  
one might consider ill-posed problems, where the existence of a root is not perturbation-stable. Another line of inquiry is the investigation of conditions under which 
a nonzero local degree of the \pl model at a root is sufficient for convergence. 
Here one might ask whether it is enough to ask that a locally open \pl model can be obtained
by perturbation of the reference point, or that a locally open model can be obtained by perturbation of the \anf corresponding to the piecewise linearization at the root (these two criteria are not equivalent). 

Finally, we would like to learn more about the structure correspondences of \pcs function 
and \pl model. The open problem stated in Remark \ref{rem:open-problem} is one example 
question. Another one concerns a key result of this paper: Can one describe a principle after which (singular) limiting Jacobians at a reference point $\rx$ are 
ignored by the tangent mode piecewise linearization $\lozenge_{\rx}F$. And, in consequence: Is it possible to classify the type of singular situations that our generalized Newton's methods can 
deal with. 

\section*{Acknowledgements}
The work for the article has been partially conducted within the Research Campus
MODAL funded by the German Federal Ministry of Education and Research
(BMBF) (fund number \(05\text{M}14\text{ZAM}\)).

\bibliography{autodiff}
\bibliographystyle{alpha}

\end{document}